\documentclass[a4paper, 10pt, DIV=14, headinclude=false, footinclude=false, leqno]{scrartcl}

\usepackage[T1]{fontenc}
\usepackage[utf8]{inputenc}
\usepackage[english]{babel}

\usepackage{amsmath,amssymb,amsthm}
\usepackage{bbm} 
\usepackage[bookmarks,
            raiselinks,
            pageanchor,
            hyperindex,
            colorlinks,
            citecolor=black,
            linkcolor=black,
            urlcolor=black,
            filecolor=black,
            menucolor=black]{hyperref}
\usepackage[capitalise, noabbrev]{cleveref}
\Crefname{algocf}{Algorithm}{Algorithms}
\usepackage{enumitem}
\usepackage{caption}
\usepackage{wrapfig}
\usepackage{sidecap} 
\usepackage{csquotes}
\usepackage{url}
\usepackage{mathtools}
\usepackage{xfrac}
\usepackage[
   algo2e,
   lined,
   boxed,
  commentsnumbered,
  linesnumbered,
  ruled
 ]{algorithm2e}
\usepackage{xparse}

\usepackage{tikz}
\usetikzlibrary{external}
\usetikzlibrary{arrows,positioning,calc,chains}

\usepackage{pgfplots}
\usepackage{pgfkeys}
\newenvironment{customlegend}[1][]{%
	\begingroup
	\csname pgfplots@init@cleared@structures\endcsname
	\pgfplotsset{#1}%
}{%
	\csname pgfplots@createlegend\endcsname
	\endgroup
}%
\def\addlegendimage{\csname pgfplots@addlegendimage\endcsname}

\usepackage{todonotes}
\usepackage{color}
\usepackage{setspace}
\definecolor{sixclassRdYlBu1}{rgb}{0.84,0.19,0.15}
\definecolor{sixclassRdYlBu2}{rgb}{0.99,0.55,0.35}
\definecolor{sixclassRdYlBu3}{rgb}{1.0,0.88,0.56}
\definecolor{sixclassRdYlBu4}{rgb}{0.88,0.95,0.97}
\definecolor{sixclassRdYlBu5}{rgb}{0.57,0.75,0.86}
\definecolor{sixclassRdYlBu6}{rgb}{0.27,0.46,0.71}

\usepackage{subcaption} 
\theoremstyle{plain}
\newtheorem{lemma}{Lemma}[section]
\newtheorem{theorem}[lemma]{Theorem}
\newtheorem{remark}[lemma]{Remark}

\newtheorem{proposition}[lemma]{Proposition}
\newtheorem{assumption}{Assumption}

\newcommand{\Params}{\mathcal{P}}
\newcommand{\Paramsad}{\mathcal{P}} 

\newcommand{\J}{\mathcal{J}}
\newcommand{\LL}{\mathcal{L}}

\newcommand{\Jhat}{\hat{\mathcal{J}}}

\newcommand{\R}{\mathbb{R}}
\newcommand{\N}{\mathbb{N}}
\newcommand{\pr}{\textnormal{pr}}
\newcommand{\du}{\textnormal{du}}

\newcommand{\red}{r}

\newcommand{\Jnoncor}{\hat{J}}
\newcommand{\cJhatn}{{{\Jhat_\red}}}

\newcommand{\cont}[1]{\gamma_{#1}}
\newcommand{\infsup}[1]{\gamma^{\text{pg}}_{#1}}
\newcommand{\Proj}{\mathrm{P}}

\DeclareMathOperator{\integralend}{d}
\newcommand{\intend}[1]{\integralend \hS{-1.25} #1}
\newcommand{\dx}{\intend{x}}

\newcommand{\kformd}{k_{\mu}}

\DeclarePairedDelimiterX\skal[2]{(}{)}{#1\,,\,#2}

\newcommand{\gridh}{\mathcal{T}_h}
\newcommand{\Gridh}{\mathcal{T}_H}

\newcommand{\Ktmu}{\mathbb{K}_{T,\mu}}

\newcommand{\hS}[1]{\hspace{#1pt}}

\DeclareMathOperator*{\argmin}{argmin}

\DeclareMathOperator{\esssup}{ess \, sup}

\DeclareMathOperator{\essinf}{ess \, inf}

\newcommand{\el}{l}
\newcommand{\kl}{\ell}
\newcommand{\uhm}[1][\mu]{u_{h,#1}}
\newcommand{\phm}[1][\mu]{p_{h,#1}}
\newcommand{\uhkm}[1][\mu]{u_{H,\kl,#1}}
\newcommand{\phkm}[1][\mu]{p_{H,\kl,#1}}

\newcommand{\uhkmsm}{u_{H,\kl,\mu}^{\text{ms}}}
\newcommand{\phkmsm}{p_{H,\kl,\mu}^{\text{ms}}}

\newcommand{\uhkmsmrb}{u_{H,\kl,\mu}^{\text{ms,rb}}}

\newcommand{\phkmsmrb}{p_{H,\kl,\mu}^{\text{ms,rb}}}

\newcommand{\uhkmrb}{u_{H,\kl,\mu}^{\text{rb}}}
\newcommand{\phkmrb}{p_{H,\kl,\mu}^{\text{rb}}}

\newcommand{\ehms}{e^{\text{ms}}_{H,\kl}}
\newcommand{\eh}{e_{H,\kl}}

\newcommand{\vf}{v^{\text{f}}}
\newcommand{\vft}[1][T]{v^\text{f}_{#1}}
\newcommand{\uft}[1][T]{u^\text{f}_{#1}}
\newcommand{\ufti}[1]{u^\text{f}_{T_{#1}}}
\newcommand{\vfti}[1]{v^\text{f}_{T_{#1}}}

\newcommand{\IH}{\mathcal{I}_H}

\newcommand{\CPG}{\infsup{\kl}}
\newcommand{\QQ}{{\mathcal{Q}}}
\newcommand{\QQm}{{\QQ_{\mu}}}
\newcommand{\QQkm}{{\QQ_{\kl,\mu}}}

\newcommand{\QQkmrb}{{\QQ_{\kl,\mu}^\text{rb}}}
\NewDocumentCommand{\QQktm}{O{T} O{\mu}}{\QQ^{#1}_{\kl,#2}}
\NewDocumentCommand{\QQktmrb}{O{T} O{\mu}}{\QQ^{#1,rb}_{\kl,#2}}
\NewDocumentCommand{\QQktmts}{O{T} O{\mu}}{\QQ^{#1,TS}_{\kl,#2}}
\newcommand{\Vfh}{V^{\text{f}}_{h}}
\newcommand{\Vfhkt}[1][T]{V^{\text{f}}_{h,\kl,#1}}
\newcommand{\Vfrbkt}[1][T]{V^{\text{f},\text{rb}}_{\kl,#1}}

\newcommand{\Tlast}{T_{\abs{\mathcal{T}_H}}}
\newcommand{\Vts}{\mathfrak{V}}

\newcommand{\Vtsrbpr}{\mathfrak{V}^{\text{rb,pr}}}
\newcommand{\Vtsrb}{\mathfrak{V}^{\text{rb}}}
\newcommand{\Vtsrblod}{\mathfrak{V}^{\text{rblod}}}
\newcommand{\Vtsrbdu}{\mathfrak{V}^{\text{rb,du}}}
\newcommand{\pts}{\mathfrak{p}}
\newcommand{\uts}{\mathfrak{u}}
\newcommand{\utsm}{\mathfrak{u}_{\mu}}
\newcommand{\ptsm}{\mathfrak{p}_{\mu}}
\newcommand{\utsmrb}{\mathfrak{u}_{\mu}^{\text{rb}}}

\newcommand{\ptsmrb}{\mathfrak{p}_{\mu}^{\text{rb}}}

\newcommand{\vts}{\mathfrak{v}}
\newcommand{\Btsm}{\mathfrak{B}_{\mu}}
\newcommand{\Btsq}[1][q]{\mathfrak{B}_{\xi}}
\newcommand{\Fts}{\mathfrak{F}}
\newcommand{\sumT}{\sum_{T\in\Gridh}}

\DeclarePairedDelimiterX\mengenA[1]{\lbrace}{\rbrace}{#1}
\DeclarePairedDelimiterX\mengenB[2]{\lbrace}{\rbrace}{#1\, \delimsize\vert \, #2}

\makeatletter
\newcommand{\set}[2][\relax]{
	\ifx#1\relax \ensuremath{
		\mengenA*{#2}}
	\else \ensuremath{%
		\mengenB*{#1}{#2}}
	\fi}
\makeatother

\DeclareRobustCommand{\minwidthbox}[2]{%
	\ifmmode
	\expandafter\mathmakebox
	\else
	\expandafter\makebox
	\fi
	[\ifdim#2<\width\width\else#2\fi]{#1}%
}

\DeclarePairedDelimiter{\abs}{|}{|}
\DeclarePairedDelimiter{\norm}{\lVert}{\rVert}

\DeclareFontFamily{U}{matha}{\hyphenchar\font45}
\DeclareFontShape{U}{matha}{m}{n}{
      <5> <6> <7> <8> <9> <10> gen * matha
      <10.95> matha10 <12> <14.4> <17.28> <20.74> <24.88> matha12
      }{}
\DeclareSymbolFont{matha}{U}{matha}{m}{n}
\DeclareFontSubstitution{U}{matha}{m}{n}

\DeclareFontFamily{U}{mathx}{\hyphenchar\font45}
\DeclareFontShape{U}{mathx}{m}{n}{
      <5> <6> <7> <8> <9> <10>
      <10.95> <12> <14.4> <17.28> <20.74> <24.88>
      mathx10
      }{}
\DeclareSymbolFont{mathx}{U}{mathx}{m}{n}
\DeclareFontSubstitution{U}{mathx}{m}{n}

\DeclareMathDelimiter{\vvvert}{0}{matha}{"7E}{mathx}{"17}

\DeclarePairedDelimiterXPP{\snorm}[1]{}{\lVert}{\rVert}{_{1}}{\ifblank{#1}{\:\cdot\:}{#1}}
\DeclarePairedDelimiterXPP{\anorm}[1]{}{\lVert}{\rVert}{_{a,\mu}}{\ifblank{#1}{\:\cdot\:}{#1}}
\DeclarePairedDelimiterXPP{\Snorm}[1]{}{\vvvert}{\vvvert}{_{1}}{\ifblank{#1}{\:\cdot\:}{#1}}
\DeclarePairedDelimiterXPP{\Anorm}[1]{}{\vvvert}{\vvvert}{_{a,\mu}}{\ifblank{#1}{\:\cdot\:}{#1}}
\DeclarePairedDelimiterXPP{\SMnorm}[1]{}{\vvvert}{\vvvert}{_{1,\mu}}{\ifblank{#1}{\:\cdot\:}{#1}}

\makeatletter
\newcommand{\labeltext}[2]{%
  \@bsphack
  \csname phantomsection\endcsname 
  \def\@currentlabel{#1}{\label{#2}}%
  \@esphack
}
\makeatother

\allowdisplaybreaks

\DeclareOldFontCommand{\rm}{\normalfont\rmfamily}{\mathrm}

\DeclareOldFontCommand{\sc}{\normalfont\scshape}{\@nomath\sc}
\title{A Relaxed Localized Trust-Region Reduced Basis Approach for Optimization of Multiscale Problems\footnote{{\textbf{Funding:}} The authors acknowledge funding by the Deutsche Forschungsgemeinschaft under Germany’s Excellence Strategy EXC 2044 390685587, Mathematics M\"unster: Dynamics -- Geometry -- Structure and by the DFG under contract OH 98/11-1.
}}

\author{Tim Keil$^\dagger$ and Mario Ohlberger\thanks{Mathematics Münster, Westfälische Wilhelms-Universität Münster, Einsteinstr. 62, D-48149 M\"unster, \url{tim.keil@uni-muenster.de, mario.ohlberger@uni-muenster.de}}}

\begin{document}
	

\maketitle

\begin{abstract}
	In this contribution, we are concerned with parameter optimization problems that are constrained by multiscale PDE state equations. As an efficient numerical solution approach for such problems, we introduce and analyze a new relaxed and localized trust-region reduced basis method. Localization is obtained based on a Petrov-Galerkin localized orthogonal decomposition method and its recently introduced two-scale reduced basis approximation. We derive efficient localizable a posteriori error estimates for the optimality system, as well as for the two-scale reduced objective functional. 
While the relaxation of the outer trust-region optimization loop still allows for a rigorous convergence result, the resulting method converges much faster due to larger step sizes in the initial phase of the iterative algorithms. 
The resulting algorithm is parallelized in order to take advantage of the localization.
Numerical experiments are given for a multiscale thermal block benchmark problem. The experiments demonstrate the efficiency of the approach, particularly for large scale problems, where methods based on traditional finite element approximation schemes are prohibitive or fail entirely.  
	
	\par\vskip\baselineskip\noindent
	\par\vskip\baselineskip\noindent
\end{abstract}


\section{Introduction}
\label{sec:introduction} 
Parameterized multiscale problems where the parameters are optimized with respect to a user-defined quality criteria are of general interest in many physical, chemical, biomedical, or engineering applications.
Examples include the optimal design of devices built from composed materials \cite{Allaire2012,ADFM2016,MR3875293}, optimization of reactive flow processes in porous media \cite{Jansen201140,MR2143505,MR3969255}, or the design of meta-materials \cite{MR3529770,Barbara18,Ohlberger202029}.
As a mathematical model for such constrained parameter optimization problems, we consider linear-quadratic parameter optimization, subject to the solution of a parameter-dependent elliptic variational multiscale problem.

The numerical approximation of such problems is computationally extremely demanding due to the multiscale character of the problem and the need for repeated PDE-solves within an outer iterative optimization loop.
Recently, substantial progress has been made, both with respect to efficient algorithms for PDE-constrained optimization and with respect to efficient model order reduction approaches for variational multiscale problems. \\

\textbf{Numerical multiscale methods and localized model order reduction.}
In the last two decades, there has been a tremendous development of suitable numerical methods for multiscale problems. The main intention is to resolve the finest required scale only locally and collect the gathered fine-scale information in an effective coarse-scale global system.
Well-established methods include the multiscale finite element method (MsFEM) \cite{MsFEM,Efendiev2009,HOS2014}, its generalized variants (GMsFEM) \cite{Efendiev2013,CEL2014}, the heterogeneous multiscale method (HMM) \cite{EEnq,EE2005,Ohl2005}, the variational multiscale method (VMM) \cite{Hughes1995387,HUGHES19983,LM2005}, the multiscale-spectral generalized finite element method (GFEM) \cite{BL2011,SP2016,BL+2020,MS22,MSD22,SS22}, and the local orthogonal decomposition (LOD) \cite{Maalqvist2014,HMP2014,LODbook}. For a recent review on multiscale methods, we refer to \cite{AHP21}.

For parameterized PDEs, model order reduction (MOR) has seen great development in the last decade \cite{MR3672144}.
A particular instance is the reduced basis method (RBM)\cite{Hesthaven2016,Quarteroni2016}.
The main idea is a splitting into an offline- and online phase.
In the offline phase, a sufficiently rich reduced basis (RB) is constructed with the full order model (FOM) that computes high-fidelity solutions.
Subsequently, a projection-based surrogate model is built from the reduced basis.
In the online phase, the resulting reduced order model (ROM) is evaluated with preferably no need to touch the high-fidelity complexity at all.
In the context of parameter optimization for multiscale problems, model order reduction can be used to accelerate repeated solves of the respective multiscale method.

Meanwhile, several applications of the RBM in the context of multiscale methods have been proposed in~\cite{Efendiev2013,HesthavenZhangEtAl2015,Nguyen2008} for the MsFEM, in~\cite{AB12,AB13,AB14_2,AB14,A19} for the HMM, in~\cite{RBLOD,tsrblod} for the LOD and in \cite{Boy2008,OS12,OS2014,SO15} for other related approaches.
In these works, the RBM has been used along with a parameterization of the local problems for speeding up the solution process of a single multiscale problem, whereas in~\cite{RBLOD, Nguyen2008,OS12,OS2014,SO15,tsrblod}, a ROM is built for each individual local problem in the context of parameterized multiscale problems.
For an overview of localized MOR and applications to parameterized multiscale problems, we refer to the review article \cite{BIORSS21}.

While the original idea of using the RBM within multiscale methods is to accelerate the solution procedures directly associated with the finest scale of the system, the resulting reduced methods are still computationally dependent on the coarse mesh size. Moreover, the global approximation error of the reduced system can often not be rigorously controlled.
Recently in \cite{tsrblod}, these issues have been resolved by an additional reduction of the coarse system that is internally based on a two-scale formulation of the multiscale scheme.\\[-.5em]

\textbf{Error aware trust-region methods for PDE-constrained optimization.}
In the context of PDE-constrained optimization, localized model order reduction with online enrichment has been suggested in \cite{OhlbergerSchaeferEtAl2018}.
The general idea of online enrichment algorithms is to specifically train the reduced models to the parameters that are queried during an optimization process.
This idea has been investigated with rigorous analysis in the context of error-aware trust-region optimization methods \cite{YM2013} and has first been used with global reduced basis approaches in \cite{QGVW2017}.
While trust-region methods in general serve the purpose of global convergence while using cheap locally accurate model functions that are only used in a mostly metric sub-region of "trust", the concept of error-aware trust-region algorithms  is to use a surrogate that allows for an error control and can be adaptively enriched along the path of optimization.
Therefore, a locally accurate surrogate model is used as long as we "trust" the surrogate model, steered by a respective a posteriori error estimator of the surrogate.
Once the boundary of the trust-region is reached and the iterate is accepted, an online enrichment is performed and the process is continued from the current iterate.
Such methods for global reduced basis methods have further been enhanced in terms of more robust algorithms and error estimation in \cite{paper1,paper2,paper3}.\\[-.6em]

We emphasize that, as far as this contribution is concerned, using (localized) MOR in an error-aware trust-region framework to solve a single PDE-constrained parameter optimization problem aims to reduce the overall computational cost of the optimization method.
Thus, neither offline- nor online computations of the approach can be considered negligible.\\[-.5em]

\textbf{Main results.}
This contribution is the first work to deviate from a FEM-based spatially global discretization for the error-aware adaptive trust-region algorithm.
We instead build on an underlying efficient localized discretization framework based on the Petrov-Galerkin LOD \cite{elf} and, as a surrogate, we use its recently introduced two-scale reduced basis approximation (TSRBLOD) \cite{tsrblod}. 
The resulting variant of a trust-region localized RB method (TR-LRB) adaptively constructs local RB models in each of the TR subproblems and deviates from a classical and potentially infeasible globally resolved FEM approximation of the underlying multiscale equation.

The application of error-aware TR methods with localized RB techniques without having to rely on a global finite-element-based discretization is an original contribution of this article. As a necessary ingredient for an efficient adaptive and localized reduced method, we derive new localized error bounds to detect where the model requires local basis updates and to find efficient global coupling techniques.
Following the ideas presented in the TSRBLOD \cite{tsrblod}, a posteriori error estimates are first derived for the primal and dual state equations of the optimality system. Based on these results, we finally obtain rigorous error bounds for the reduced objective functional needed for the TR algorithm.

As a further original contribution, we introduce a relaxed version of the basic TR algorithm that allows for larger step sizes in the initial iterations of the TR-optimization loops without sacrificing the provable convergence of the overall method.
The relaxation of the TR algorithm is applicable for both global FEM-based and localization-based surrogates in the TR method.
The global convergence of the resulting relaxed TR algorithm is stated in Theorem \ref{Thm:convergence_of_TR}.
Notably, such an approach is not restricted to RB-based methods and is also useful if the respective surrogate suffers from a poor or expensive error estimator. 

Finally, we provide numerical experiments that demonstrate the applicability of our approach for large scale optimization problems where global FEM approaches are not feasible anymore. \\[-.5em]

\textbf{Organization of the article.}
The article is organized as follows: In \cref{sec:problem}, we detail the mathematical formulation of the considered multiscale optimization problem. In \cref{sec:relaxed}, we discuss the general formulation of the relaxed error-aware adaptive TR algorithm.
Subsequently, in \cref{sec:tsrblod}, we show that the TSRBLOD can be used as an instance of the TR algorithm.
Lastly, in \cref{sec:loc_experiments}, we present numerical experiments that demonstrate the benefit of localized techniques.

\section{Parameter optimization of multiscale problems}
\label{sec:problem}

In this work, we are concerned with the efficient approximation of linear-quadratic parameter optimization, subject to a parameter-dependent multiscale variational state equation, which is typically 
given as a weak formulation of an underlying elliptic multiscale PDE. 

To this end, let $V$ be a real-valued Hilbert space and
let $\Params \subset \R^P$, with $P \in \N$ denote a compact and convex admissible parameter set, given by box constraints of the form 
$
	\Params:= \left\{\mu\in\mathbb{R}^P\,|\,\mu_\mathsf{a} \leq \mu \leq \mu_\mathsf{b} \right\} \subset \R^P,
$
for given parameter bounds $\mu_\mathsf{a},\mu_\mathsf{b}\in\mathbb{R}^P$, where {``$\leq$''} has to be understood component-wise. 

Let $\J\colon V \times \Params \to \R^{>0}$ be a 
continuous functional
(an explicit example is given in \eqref{eq:L2_misfit} and further assumptions are posed in \cref{sec:a_posteriori_estimates}).
We seek a local solution to the following PDE-constrained optimization problem:

{\color{white}
	\begin{equation}
	\tag{P}
	\label{P}
	\end{equation}
}\vspace{-45 pt} 
\begin{subequations}\begin{align}
	\hS{130}
	& \min_{\mu \in \Params} \J(u_\mu, \mu),
	\tag{P.a}\label{P.argmin}\intertext{%
		subject to $u_\mu \in V$ being the solution of the \emph{state -- or primal -- variational equation}
	}
	&a_\mu(u_\mu, v) = l_\mu(v) & \hspace*{-2em}\text{for all } v \in V.
	\tag{P.b}\label{P.state}
	\end{align}\end{subequations}%
\setcounter{equation}{0}
For each admissible parameter $\mu \in \Params$, $a_\mu: V \times V \to \R$ denotes a continuous and coercive bilinear form and~$l_\mu: V \to \R$ is a continuous linear functional.

We introduce the reduced cost functional $\Jhat: \Params\to \mathbb{R},\,\mu\mapsto \Jhat(\mu) := \J(u_\mu, \mu)= \J( \mathcal S(\mu),\mu)$, where $S: \Params \to V$ is the parameter to solution map of~\cref{P.state}. Then, \eqref{P} is equivalent to the so-called reduced problem
\begin{align}
\min_{\mu \in \Params} \Jhat(\mu).
\tag{\textnormal{RP}}
\label{Phat}
\end{align}
\par
We are particularly interested in multiscale applications in the sense that the parameter-dependent bilinear form $a_\mu$ involves spatial heterogeneities. Throughout this contribution, we will consider the case of elliptic multiscale problems, where $V = H^1_0(\Omega), \Omega \subset \R^d$.
Furthermore, $a_\mu$ and $l_\mu$ are given as 
\[
a_\mu(u_\mu, v) := \int_\Omega A_\mu(x) \nabla u_\mu(x), \nabla v(x) \dx, \qquad \text{and} \qquad l_\mu(v) := \int_\Omega f_\mu(x) v(x) \dx,
\]
where the family of diffusion or conductivity tensors $A_\mu: \Omega \to \R^{d\times d}$ have a rich multiscale structure that would lead to very high dimensional approximation spaces for the state space when approximated, e.g., with classical finite element type methods. For an example, we refer to Fig.~\ref{fig:thermal_block} below for particular choices of such multiscale conductivity fields.
Moreover, we let $f_\mu \in L^2(\Omega)$. 
We employ standard assumptions on the underlying multiscale PDE.
In particular, $A_{\mu} \in L^{\infty}(\Omega, \mathbb{R}^{d \times d})$ to be symmetric and uniformly elliptic, such that
\begin{align} 
0 < \alpha := \essinf\limits_{x \in \Omega} \inf_{v \in \mathbb{R}^d \setminus \set{0}} \frac{\left( A_\mu(x)v \right) \cdot v}{ v \cdot v} \qquad \text{and} \qquad
\infty > \beta := \esssup\limits_{x \in \Omega} \sup_{v \in \mathbb{R}^d \setminus \set{0}} \frac{\left( A_\mu(x)v \right) \cdot v}{ v \cdot v}.
\end{align}
Further, we let $\kappa:=\beta/\alpha$ be the maximum contrast of $A_{\mu}$ for all $\mu \in \Params$. Moreover, For $v\in V = H^1_0(\Omega)$, we define the standard (equivalent) norms:
\begin{equation*}
	\snorm{v}^2       := \int_\Omega \abs{\nabla v(x)}^2\dx, \qquad\qquad
	\anorm{v}^2       := \int_\Omega \abs{A_\mu^{1/2}\nabla v(x)}^2\dx.
\end{equation*}
Note that $\snorm{}$ is a norm on $V$ due to Friedrich's inequality, and $\anorm{}$ denotes the parameter-dependent energy norm.
We emphasize that the homogeneous Dirichlet boundary conditions and the symmetry of $A_\mu$ are assumed for simplicity and to avoid technicalities in the definition and analysis of the TR-LRB method below. 
All concepts elaborated in this paper can, however, be generalized to more complex underlying PDEs.\par

As usual in the context of reduced basis methods, we require parameter separability for $a_{\mu}$, $l_{\mu}$, 
and $\J$
for an efficient offline-online decomposition.
However, in many cases, this assumption needs to be artificially constructed with the help of the so-called Empirical Interpolation (EI) \cite{BarraultMadayEtAl2004,CS2010,DHO2012,CEGG2014}. 
Furthermore, to derive optimality conditions for \eqref{P}, we require sufficient regularity of the linear and bilinear forms, as well as the objective functional w.r.t. the parameter. 

We define for given $u \in V$, $\mu \in \Params$, the primal residual $r_\mu^\pr(u) \in V'$ associated with \eqref{P.state} by
\begin{align}
r_\mu^\pr(u)[v] := l_\mu(v) - a_\mu(u, v) &&\text{for all }v \in V.
\label{eq:primal_residual}
\end{align}

To find a solution for \eqref{P}, we follow the \emph{first-optimize-then-discretize} approach, i.e., we deduce the first-order necessary optimality system by considering the Lagrangian functional $\LL(u,\mu,p) = \J(u,\mu) + r_\mu^\pr(u)[p]$ and taking its derivative to all variables.
Then, there exists an associated unique Lagrange multiplier $\bar p\in V$, such that 
\begin{subequations}
	\label{eq:optimality_conditions}
	\begin{align}
	r_{\bar \mu}^\pr(\bar u)[v] &= 0 &&\text{for all } v \in V,
	\label{eq:optimality_conditions:u}\\
	\partial_u \J(\bar u,\bar \mu)[v] - a_\mu(v,\bar p) &= 0 &&\text{for all } v \in V,
	\label{eq:optimality_conditions:p}\\
	(\partial_\mu \J(\bar u,\bar \mu)+\nabla_{\mu} r^\pr_{\bar\mu}(\bar u)[\bar p]) \cdot (\nu-\bar \mu) &\geq 0 &&\text{for all } \nu \in \Params.
	\label{eq:optimality_conditions:mu}
	\end{align}	
\end{subequations}
The tuple $(\bar u, \bar \mu) \in V \times \Paramsad$ is called a first-order critical (FOC) point.
We note that the existence of a stationary point can be shown (given mild assumptions), but uniqueness is not necessarily given.
For further details on the optimality system, we refer to \cite[Cor. 1.3]{HPUU2009}. 

While \eqref{eq:optimality_conditions:u} restates the state equation \eqref{P.state}, from \eqref{eq:optimality_conditions:p}, we observe the so-called \emph{adjoint-} or \emph{dual} equation  
\begin{align}
a_\mu(q, p_\mu) = \partial_u \J(u_\mu, \mu)[q]
&&\text{for all } q \in V,
\label{eq:dual_solution}
\end{align}
with solution $p_{\mu} \in V$ for a fixed $\mu \in \Params$ and given the solution $u_\mu \in V$ to the state equation \eqref{P.state}.
For given $u, p \in V$, we introduce the dual residual $r_\mu^\du(u, p) \in V'$ associated with \eqref{eq:dual_solution} as
\begin{align}
r_\mu^\du(u, p)[q] :=  \partial_u \J(u_\mu, \mu)[q] 
- a_\mu(q, p) &&\text{for all }q \in V.
\label{eq:dual_residual}
\end{align}
The dual solution $p_{\mu}$ is of great interest for computing the gradient of the objective functional~$\Jhat$ efficiently, i.e.
$$\nabla_\mu \Jhat(\mu) = \nabla_\mu \J(u_{\mu}, \mu) + \nabla_\mu r_{\mu}^\pr(u_\mu)[p_{\mu}].$$
We also note that sufficient conditions based on the Hessian can be used to verify that the stationary point is indeed a minimum; cf.~\cite{paper1}.
However, in our algorithms, this condition can always be checked a posteriori.

\section{Relaxed trust-region method}
\label{sec:relaxed}

In this section, we introduce a relaxed variant of the error-aware adaptive reduced TR algorithm that has been initially proposed in \cite{YM2013} and has so far been applied to global RB-based surrogates in \cite{QGVW2017,paper1,paper2}.
The proposed relaxation can be used as a warm start of the certified TR algorithm and, importantly, fulfills the same general convergence result as presented in \cite{paper2}.
Notably, the relaxation technique is entirely independent of the choice of the specific surrogate model and is helpful for both FEM-based as well as localized methods, cf.~\cref{sec:loc_experiments}.
We emphasize that there are many reasons why we propose a relaxation of the original algorithm.~
More details are given throughout this section and in \cref{sec:loc_experiments}.
In the sequel, we carefully introduce the new variant and relate it to the original algorithm without relaxation from \cite{paper1,paper2}  further below.

Let~$\Jhat_h$ be a sufficiently accurate discrete version of $\Jhat$, meaning that the primal and dual equations above are discretized with an appropriate, most likely high-dimensional, finite-dimensional space $V_h$ with fine mesh~$\gridh$.
We call this discrete reference model the full order model (FOM) and make the following assumption:
\begin{assumption}[The FOM is the ``truth'']
	\label{asmpt:truth}
	We assume that the FOM discretization error $| \Jhat(\mu) - \Jhat_h(\mu)|$ can be neglected,
	which also translates to the primal and dual discretization error being negligible.
\end{assumption}

At this point we do not specify the concrete discretization scheme for computing~$\Jhat_h$.
In the already existing works \cite{paper1,paper2,QGVW2017}, this is based on a standard FEM approximation.
However, as we see later, we may also choose a multiscale method as the underlying FOM method.

Let $\Jhat_\red$ denote a surrogate of the objective functional $\Jhat_h$, which is obtained by replacing the FOM with a reduced order model (ROM).
Further, we assume that the surrogate $\Jhat_\red$ admits an a posteriori error result, such that
\begin{equation}
	| \Jhat_h(\mu) - \Jhat_\red(\mu) | \leq \Delta_{\Jhat_r}(\mu),
\end{equation}
where $\Delta_{\Jhat_r}(\mu)$ can be computed without explicitly evaluating $\Jhat_h$ at $\mu \in \Params$.

While the algorithm is not restricted to it, in  \cite{paper1,paper2,QGVW2017}, global RB methods have been used, which is why the algorithm is often abbreviated as the TR-RB method, cf.~\cref{rem:generality}.
We refer to \cite{paper1,paper2} for corresponding a posteriori error estimation results in the global RB case.
A particular variant of a localized RB (LRB) approach will be discussed and analyzed in \cref{prop:LOD_reduced_functional}, where also respective error estimates are derived.

The relaxed error-aware adaptive TR algorithm iteratively computes a first-order critical (FOC) point of problem \eqref{P} and can be divided into an outer- and multiple inner optimization procedures.
Roughly speaking, every outer iteration includes costly FOM evaluations, an enrichment phase, and a cheap inner sub-problem.

In a standard TR method, in each outer iteration $k\geq 0$ of the TR, a local sub-problem is solved, minimizing a cheap model function $m^{(k)}$ for the expensive objective functional $\J_h$ that is only considered valid in the so-called trust-region (for instance determined by a metric distance).
In our adaptive TR approach, we use
$$m^{(k)}(\cdot):= \cJhatn^{(k)}(\mu^{(k)}+\cdot \,),$$
i.e., the locally accurate surrogate model, adaptively enriched along the outer path of the TR approach (indicated by the subscript $k$).
In particular, unlike a standard TR approach, we always reuse the enriched surrogate, which we only initialize with the initial guess $\mu^{(0)} \in \Params$ at $k=0$.

In our method, we solve the sub-problem in the error-aware trust-region with radius $\delta^{(k)}$, characterized by the a posteriori error estimator of the surrogate and relaxed by a corresponding relaxation parameter $\varepsilon^{(k)}$.
To this end, let $(\varepsilon^{(k)})_k$ denote an a priori chosen null sequence, where we assume the existence of $K \in \mathbb{N}$ such that $\varepsilon^{(k)} = 0$ for all $k > K$.
We solve for a solution $\bar s \in \Params$ of
\begin{equation}
	\label{eq:rel_TRRBsubprob}
	\min_{s\in\Paramsad} \cJhatn^{(k)}(\mu^{(k)} + s) \qquad \text{ s.t. } \qquad \frac{\Delta_{\Jhat_\red^{(k)}}(\mu^{(k)} + s)}{\cJhatn^{(k)}(\mu^{(k)} + s)}\leq \delta^{(k)} + \varepsilon^{(k)},
\end{equation}
and set $\mu^{(k+1)} := \mu^{(k)}+ \bar s$. Hence, unlike in the previous works, the relaxation allows for larger steps in the sub-problem, which is particularly of interest if the radius has been chosen inappropriately small.

We solve \cref{eq:rel_TRRBsubprob} with an inner optimization routine and, at inner iteration $\el$, we set
\begin{align}
	\label{eq:General_Opt_Step_point}
	\mu^{(k,\el)}(j):= \Proj_{\Paramsad}(\mu^{(k,\el)} + \kappa^j d^{(k,\el)}) \in{\Paramsad} && \textnormal{for } j\geq 0,
\end{align}
where $\kappa\in(0,1)$ and $\Proj_\Params$ is a projection on the admissible parameter space. Furthermore, $d^{(k,\el)}\in\mathbb{R}^P$ is a descent direction at the iteration $(k,\el)$, computed, e.g., by the projected BFGS, reported in \cite[Section 5.5.3]{kelley}.
It has been shown in \cite{paper2} that a projected Newton algorithm for the sub-problems can enhance the convergence speed and accuracy of the inner sub-problem which we do not consider for simplicity.
Furthermore, we enforce an Armijo-type condition for inequality constraints 
\begin{equation}
	\label{Armijo}\cJhatn^{(k)}(\mu^{(k,\el)}(j)) - \cJhatn^{(k)}(\mu^{(k,\el)}) \leq  -\frac{\kappa_{\mathsf{arm}}}{\kappa^j} \| \mu^{(k,\el)}(j)-\mu^{(k,\el)}\|^2_2,
\end{equation}
with $\kappa_{\mathsf{arm}}=10^{-4}$, combined with the trust-region constraint in \eqref{eq:rel_TRRBsubprob}.

We terminate with a standard reduced FOC termination criteria, modified with $\Proj_\Params$ to account for constraints on the parameter space as proposed in \cite{paper1}:
\begin{subequations}\label{Termination_crit_sub-problem}
	\begin{equation}
	\label{FOC_sub-problem}
	\big\|\mu^{(k,\el)}-\Proj_{\Paramsad}(\mu^{(k,\el)}-\nabla_\mu \Jhat_\red^{(k)}(\mu^{(k,\el)}))\big\|_2\leq \tau_\text{{sub}},
	\end{equation}
    where $\mu^{(k,\el)} := \mu^{(k)}+ s^{(\el)}$.
	Additionally, we use a second boundary termination criterion to prevent the sub-problem from spending too much computational time on the boundary of the (relaxed) trust-region.
	\begin{equation}
		\label{eq:rel_cut_of_TR}
		\beta_2(\delta^{(k)}+ \varepsilon^{(k)}) \leq \frac{\Delta_{\Jhat_\red^{(k)}}(\mu^{(k,\el)})}{\Jhat_\red^{(k)}(\mu^{(k,\el)})} \leq \delta^{(k)}+ \varepsilon^{(k)}.
	\end{equation}
\end{subequations}
\noindent where $\tau_\text{{sub}}\in(0,1)$ is a predefined tolerance and $\beta_2\in(0,1)$, generally close to one.

After the outer iterate $\mu^{(k+1)} := \mu^{(k,L)}$ of \eqref{eq:rel_TRRBsubprob} has been computed in $L$ inner optimization steps, the (relaxed) sufficient decrease condition helps to decide whether to accept the iterate:
\begin{align}
	\label{eq:suff_decrease_condition_rel}
	\Jhat_h^{(k+1)}(\mu^{(k+1)})\leq \cJhatn^{(k)}(\mu_\text{{AGC}}^{(k)}) +
	\varepsilon^{(k)}, 
\end{align}
where $\mu_\text{AGC}^{(k)}$ denotes the approximated generalized Cauchy point, in our case, the first (gradient-descent) step of the sub-problem.
We note that this condition potentially allows for a step that increases the functional by a factor of $\varepsilon^{(k)}$.
Condition \eqref{eq:suff_decrease_condition_rel} can be cheaply checked by using a sufficient and necessary condition, cf. \cite{paper1,YM2013}.
However, if the cheap conditions are not applicable, we check~\eqref{eq:suff_decrease_condition_rel} explicitly.
If the iterate is rejected, we shrink the TR-radius and repeat the sub-problem.
If, instead, $\mu^{(k+1)}$ is accepted, we use the parameter to enrich the reduced model.
We emphasize that the relaxation of \eqref{eq:suff_decrease_condition_rel} is substantial for the relaxed method.
Without a relaxation \eqref{eq:suff_decrease_condition_rel}, the method may disregard many of the outer iteration points, which would contradict the relaxation of the sub-problem and may result in a very slow method.
The potential rejection and the corresponding extra FOM effort to adjust the TR radius in the early stages of the algorithm is indeed one of the main reasons for the relaxation of the original method.
That is, the choice of a "perfect" TR-radius and shrinking- or enlarging factors is problem-dependent and, at least to our knowledge, always has to be found by (computationally demanding) trials.

Overall convergence of the algorithm can be verified with a FOM-based FOC condition
\begin{equation} \label{eq:global_termination_TR}
\|\mu^{(k+1)}-\Proj_\Params(\mu^{(k+1)}-\nabla_\mu \Jhat_h(\mu^{(k+1)}))\|_2\leq \tau_{\text{\rm{FOC}}},
\end{equation}
where the FOM quantities are available from the enrichment. Moreover, this allows to compute a condition for possible enlargement of the TR radius if the reduced model is better than expected, cf.~\cite{paper1}.
We also mention that a reduced Hessian can be used for an a posteriori post-processing for the optimal parameter, cf.~\cite{paper2}.

\begin{figure}[t]
	\centering
	\begin{tikzpicture}[>=stealth]
	\node[inner sep=0pt, opacity=.5] (thirdTR) at (0.22,0.22)
	{\includegraphics[width=.325\textwidth]{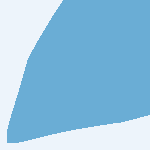}};
	\node[inner sep=0pt, opacity=.4] (secondTR) at (0.22,0.22)
	{\includegraphics[width=.325\textwidth]{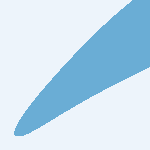}};
	\node[inner sep=0pt, opacity=.3] (firstTR) at (0.22,0.22)
	{\includegraphics[width=.325\textwidth]{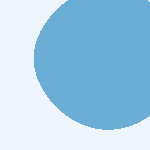}};
	
	\draw[black](2.3,-2.1) node[right] {$\mathcal{P}$};
	\draw[black](-0.3,2.) node[right] {$\textcolor{black}{\Delta_{
				\hat{\mathcal{J}}_r^{(k)}}}$};
	\draw[black](1.5,1.) node[right] (scrip) {$\mu^{(0)}$};
	\draw[fill=black] (1.5,1.) circle [radius=0.03, color=black];
	
	\draw[->,black!40, dashed] (1.5,1.) -- (-0.55,-0.7);
	\draw[->,black] (1.5,1.) -- (1.6,0.35);
	\draw[fill=black!40] (1.6,0.32) circle [radius=0.02, color=black!40];
	\draw[black](1.6,0.32) node[right] (scrip) {\scriptsize \textcolor{black!60}{$\mu^{(0, 1)}$}};
	\draw[->,black] (1.6,0.32) -- (1.25,-0.6);
	\draw[fill=black!40] (1.22,-0.6) circle [radius=0.02, color=black!40];
	\draw[black](1.22,-0.6) node[right] (scrip) {\scriptsize \textcolor{black!60}{$\mu^{(0, 2)}$}};
	\draw[->,black] (1.22,-0.6) -- (0.9,-0.9);
	\draw[fill=black!40] (0.87,-0.9) circle [radius=0.02, color=black!40];
	\draw[black](0.87,-0.9) node[below] (scrip) {\qquad\scriptsize \textcolor{black!60}{$\mu^{(0, 3)}$}};
	\draw[->,black] (0.87,-0.9) -- (0.2,-1.);
	\draw[fill=black!40] (0.17,-1.) circle [radius=0.02, color=black!40];
	\draw[black](0.17,-1.) node[below] (scrip) {\quad\scriptsize \textcolor{black!60}{$\mu^{(0, 4)}$}};
	\draw[->,black] (0.17,-1.) -- (-0.55,-0.75);
	
	\draw[black](-0.6,-0.75) node[below] (scrip) {\,$\mu^{(1)}$};
	\draw[fill=black] (-0.6,-0.75) circle [radius=0.03, color=black];
	
	\draw[->,black!40, dashed] (-0.6,-0.75) -- (-1.15,-0.65);
	\draw[fill=black](-1.2,-0.65) circle [radius=0.03, color=black];
	
	\draw[->,black!40, dashed] (-1.25,-0.65)-- (-1.6,-0.75);
	\draw[fill=black](-1.65,-0.75) circle [radius=0.03, color=black];
	\draw[black](-1.65,-0.75) node[left] (scrip) {$\bar{\mu}$};
	
	\draw[black](1.2,-1.7) node[right] {\scriptsize \textcolor{black!70}{$\delta^{(0)}$}};
	\draw[black](-1.7,-1.7) node[right] {\scriptsize \textcolor{black!70}{$\delta^{(1)}$}};
	\draw[black](-0.7,-1.8) node[right] {\scriptsize \textcolor{black!70}{$\delta^{(2)}$}};
	
	\draw (-2.5,-2.5) rectangle (2.98,2.98);
	\end{tikzpicture}
	\caption{Illustration of the error-aware adaptive TR algorithm (without relaxation). From the starting parameter $\mu^{(0)}$, several inner optimization steps (solid black arrows) are performed until the sub-problem reaches the boundary of the TR, constrained by $\Delta_{\Jhat_\red^{(0)}}< \delta^{(0)}$. For the second outer iteration (dashed grey arrows), the TR radius for the sub-problem is shrunk. Already after the third outer iteration, the optimum lies in the interior of the trust-region.}
	\label{fig:illustration}
\end{figure}
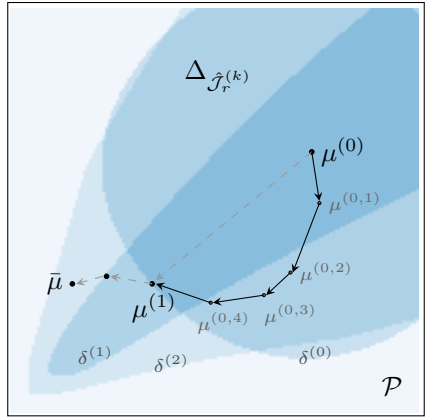
\begin{algorithm2e}[!h]
	\KwData{
	initial guess $\mu^{(0)}$, initial radius $\delta^{(0)}$, null sequence $(\varepsilon^{(k)})_k$, and tolerances $\beta_2$, $\tau_\textnormal{{sub}}$, $\tau_\textnormal{{FOC}}$.
	}
	Initialize RB model with $\mu^{(0)}$ and set $k=0$\; 
	\While
	{not \eqref{eq:global_termination_TR}}{
		Compute $\mu^{(k+1)}$ as solution of~\eqref{eq:rel_TRRBsubprob} with termination criteria~\eqref{Termination_crit_sub-problem}\label{solve_sub_problem}\;
		\uIf{Sufficient decrease condition~\eqref{eq:suff_decrease_condition_rel} is fulfilled with relaxation $\varepsilon^{(k)}$\label{suff_dec}} {
			Accept $\mu^{(k+1)}$ and
			enrich the RB model at $\mu^{(k+1)}$\; 
			Possibly enlarge the TR-radius \label{accept}\; 
		}
		\Else {
			Reject $\mu^{(k+1)}$, shrink the TR radius $\delta^{(k)}$ and go to Line \ref{solve_sub_problem}\;		
		}
		Set $k=k+1$\;
	}
	\caption{Basic TR-RB algorithm} 
	\label{alg:basic_TRRB}
\end{algorithm2e}

To conclude, \cref{alg:basic_TRRB} summarizes the main steps of the algorithm.
We note that in lines \ref{suff_dec} and \ref{accept} of  \cref{alg:basic_TRRB}, we have neglected detailed information on the exact computational procedure concerning the cheap conditions for the sufficient decrease conditions according to \cite{YM2013} and the enlarging of the TR-radius with a suitable accessible condition.
For both features, we again refer to~\cite{paper1}. In \cref{fig:illustration}, we illustrate the described procedure for a simple example with a two-dimensional parameter space.

\begin{remark}[Choice of the surrogate] \label{rem:generality}
	We emphasize that the above presented relaxed algorithm can be used for arbitrary surrogates that enable a corresponding error control and the convergence can be shown, cf. \cref{Thm:convergence_of_TR}.
	The variants in former works were abbreviated by the TR-RB algorithm, where RB explicitly refers to the (global) reduced basis reduction.
	In order to underline the generality, we avoided the explicit use of this abbreviation.
\end{remark}

\begin{remark}[Equivalence of the error aware R-TR and TR]
	Independent of the reduction approach, we also emphasize that the introduced relaxed TR approach is (apart from the sub-problem solver and minor specifics) equivalent to the algorithm from \cite{paper1} and \cite{paper2} if $\varepsilon^{(k)} = 0$.
	Thus, the relaxed TR algorithm can be interpreted as the original TR algorithm with a warm start.
\end{remark}

Concerning the convergence, we can reuse the convergence result of the TR algorithm from \cite[Theorem 3.8]{paper2}, formulated in the following theorem.
\begin{theorem}[Convergence of the relaxed error aware adaptive TR algorithm]
	\label{Thm:convergence_of_TR}
	Let $(\varepsilon^{(k)})_k$ be a null sequence as defined above. Let sufficient assumptions on the Armijo search to solve \eqref{eq:rel_TRRBsubprob} be given; cf.~\cite{paper2}.
	Then every accumulation point $\bar\mu$ of the sequence $\{\mu^{(k)}\}_{k\in\mathbb{N}}\subset \Params$ generated by the described R-TR-RB algorithm is an approximate first-order critical point for $\Jhat_h$, i.e., it holds
	\begin{equation}
		\label{First-order_critical_condition}
		\|\bar \mu-\Proj_\Params(\bar \mu-\nabla_\mu \Jhat_h(\bar \mu))\|_2 = 0.
	\end{equation}
\end{theorem}
\begin{proof}
	Since $\varepsilon^{(k)} = 0$ for all $k > K$, we can consider the result of the $K$ iterations of the R-TR as a warm start for the TR.
	Hence, let $\mu^{(K)}$ be the initial guess of the TR algorithm. Then, the convergence theorem  
	\cite[Theorem~3.8]{paper2} can be used and automatically holds for the R-TR.
\end{proof}

We also note that the convergence result is independent of the specific sub-problem solver, e.g., projected Newton (as in \cite{paper2}) or projected BFGS (as used below).
As with all different optimization methods for the discussed optimization problems, it can always happen that the R-TR algorithm finds a different local minimum than the originally proposed TR algorithm.

\section{Optimization of multiscale problems based on the TSRBLOD}
\label{sec:tsrblod}

One of the main contributions of this article is to equip the above-introduced relaxed error-aware adaptive TR method with a localized reduced basis approach such that costly global evaluations of \eqref{P.state} are not required anymore.
Our approach is based on the TSRBLOD from \cite{tsrblod}, which has recently been introduced for the Petrov--Galerkin version of the localized orthogonal decomposition method (PG--LOD) \cite{elf}.
The (relaxed) TR approach that we detailed in the previous section can naturally be used for the localized case simply by using appropriate choices for the discrete functional $\Jhat_h$ and its reduced version $\Jhat_\red$.
Specific circumstances regarding the initial construction and enrichment of the localized surrogate may occur.

In the following, we give a precise definition of a localized reduced FOM functional~$\Jhat_h^{\text{loc}}$, the localized reduced ROM functional~$\Jhat_\red^{\text{loc}}$, and its gradients~$\nabla\Jhat_h^{\text{loc}}$, and~$\nabla\Jhat_\red^{\text{loc}}$.
Furthermore, we elaborate on the error estimator~$\Delta_{\Jhat_r^{\text{loc}}}$ and provide details on the localized online enrichment of the ROM.

\subsection{Localized full order model using the LOD}
\label{sec:localized_FOM}

The LOD is a well-established multiscale method that is flexible for rough and non-periodic multiscale coefficients.
Since we use the PG--LOD as the FOM model of the optimality system \eqref{eq:optimality_conditions}, we now explain the primary concepts of the method.
We refer to \cite{LODbook} for more background and to \cite{tsrblod}, where the same notation is used.\\

\textbf{Localized orthogonal decomposition method.}
As typical for multiscale methods, we use a low-dimensional coarse-mesh $\Gridh$ with mesh size $H \gg h$ such that $\gridh$ is a refinement of $\Gridh$. We construct the respective FE space by $V_H := V_h \cap \mathcal{P}_1(\Gridh)$.
The corresponding ideal LOD space is defined by 
\begin{equation*}
V_{H, \mu}^{\text{ms}} := (I - \QQm) (V_H).
\end{equation*}
Here, the fine-scale corrections $\QQm(v) \in \Vfh$, for a given $v_h \in V_h$, are the solution of
\begin{equation}\label{eq:corrector_problem}
a_{\mu}(\QQm(v_h),\vf) = a_{\mu}(v_h,\vf) \qquad\qquad \text{for all } \; \vf \in \Vfh,
\end{equation}
where the fine-scale space $\Vfh := \ker(\IH) \subset V_h$ can be obtained with an interpolation operator $\IH : V_h \to V_H$ that maps a high-fidelity function $v_h \in V_h$ to the coarse FE space $v_H \in V_H$.
In conclusion, $\QQm(v_h)$ is the $a_\mu$-orthogonal projection of $v_h$ onto $\Vfh$, such that we have $a_\mu$-orthogonal splitting of $V_h$, i.e., $V_h = V_{H,\mu}^{\text{ms}} \oplus_{a_{\mu}} \Vfh$.

Since both spaces are still defined on the whole computational domain, we use corresponding truncated fine-scale correctors $\QQktm(v) \in \Vfhkt$, where $\Vfhkt: = \Vfh \cap H^1_0(U_\kl(T))$ is the localized fine-scale space on a coarse-scale patch $U_\kl(T)$. Here, $U_\kl(T) := U(U_{\kl-1}(T))$ is defined recursively with $U_1(T) :=  U(T)$
and $U(S), S \subset \Omega$ denotes the union of all elements in $\Gridh$ that intersect $S$.
Hence, we solve locally on $U_\kl(T)$
\begin{align} \label{eq:loc_cor_problems}
a_{\mu}(\QQktm(v_h),\vf) = a^T_{\mu}(v_h, \vf) \qquad \qquad \text{for all }\; \vf \in \Vfhkt,
\end{align}
where $a^T_{\mu}$ denotes the bilinear form obtained by restricting the integration domain in the definition of $a_{\mu}$ to $T \in \Gridh$. 
The resulting localized space can then be defined as
\[
V_{H,\kl,\mu}^{\text{ms}} := (I - \QQkm)(V_H),
\]
where $\QQkm := \sum_T \QQktm$ contains all localized corrector functions $\QQktm$ from \eqref{eq:loc_cor_problems}.

Finally, we approximate the solution of \eqref{P.state} by the Petrov--Galerkin version of the LOD: Find $\uhkmsm \in V_{H,\kl,\mu}^{\text{ms}}$, such that
\begin{equation}
\label{eq:PG_LOD}
	a_{\mu}(\uhkmsm,v_H) =  l_\mu(v_H) \qquad \text{for all } \; v_H \in V_H.
\end{equation}

We note that the standard Galerkin formulation can be obtained by using $V_{H,\kl,\mu}^{\text{ms}}$ also as the test function; see \cite{LODbook}.
To ensure that~\eqref{eq:PG_LOD} has a unique solution, we require inf-sup stability of $a_\mu$ w.r.t.\ $V_{H,\kl,\mu}^{\text{ms}}$ and $V_H$. 
As proposed in \cite{tsrblod}, an appropriate inf-sup stability constant is given by
\begin{equation} \label{eq:infsup_constant}
\CPG := \inf_{0\neq w_H\in V_H}\sup_{0\neq v_H\in V_H}
\frac{a_\mu(w_H-\QQktm(w_H), v_H)}{\anorm{w_H-\QQktm(w_H)}\snorm{v_H}} > 0,
\end{equation}
for all $l>l_0$, $l_0 \in \N$. The proof of the inf-sup stability is conditioned on sufficiently large~$\kl$; cf.~\cite{tsrblod, elf, HM19}.\\

\textbf{A priori error estimate of the PG--LOD.}
Writing the solution of~\eqref{eq:PG_LOD} as $\uhkmsm=\uhkm - \QQkm(\uhkm)$ with $\uhkm \in V_H$,
we have the following a priori estimate, which was first shown in~\cite{elf}.
\begin{theorem}[A priori convergence result for the PG-LOD]
	\label{thm:pglod_convergence}
	For a fixed parameter $\mu \in \Params$, let 
	$u_{h,\mu} \in V_h$ be the finite-element solution of \eqref{P.state} given by
	$
	a_\mu(u_{h,\mu}, v_h) = l(v_h), \text{ for all } v_h \in V_h.
	$
	Then, it holds that
	\begin{equation*}
		\norm{u_{h,\mu} - u_{H,\kl,\mu}}_{L^2} + \norm{u_{h,\mu} - u_{H,\kl,\mu}^{\text{ms}}}_{1} \lesssim (H + \theta^\ell
		\kl^{d/2}) \norm{f}_{L^2(\Omega)},
	\end{equation*}
	with $0 < \theta < 1 $ independent of $H$ and $\kl$, but dependent on the contrast $\kappa = \beta/\alpha$.
	The result, in particular, follows from the exponential decay of the fine-scale correctors for all $v_H \in V_H$:
	\begin{equation}\label{eq:exponential_decay}
		\anorm{(\QQm - \QQkm)(v_H)} \lesssim \theta^\ell \ell^{d/2} \anorm{v_H}.
	\end{equation}
\end{theorem} 
For a detailed discussion on the decay variable $\theta$, we refer to~\cite{MH16, HM19, MP14}.
We emphasize that the LOD is generally vulnerable to high-contrast problems or rapid coarse-scale changes induced by high conductivity channels since $\theta$ depends on the contrast of the problem.
For neglecting the issue of high contrast in the LOD, the interpolation operator $I_H$ has to be adjusted.
For instance, works in this direction have been done in \cite{MH16,brown2016multiscale,peterseim2016robust}.
Note that using a right-hand-side correction as in~\cite{HKM20,HM19}, for instance, can further enhance \cref{thm:pglod_convergence} by removing the $H$ dependency but requires additional corrector problems.

\begin{remark}[Fulfillment of \cref{asmpt:truth}]
	For \cref{asmpt:truth}, we may assume that an appropriate choice of the coarse-mesh size $H$, fine-mesh size $h$, and localization parameter $\kl$ is given to cope with the underlying problem.
	This means that the LOD errors $\norm{\uhm - \uhkm}_{L^2}$, $\norm{\phm - \phkm}_{L^2}$, and the corresponding $H^1$-errors are negligible, where $\uhm,\phm \in V_h$ denote the FE solution of the primal and dual equation, respectively.
\end{remark}

\textbf{Approximation of the objective functional.}
The corresponding primal variable can now be used to compute the corresponding localized FOM objective functional, i.e.
\begin{equation}
	\Jhat_h^{\text{loc}}(\mu) := \J(\uhkm, \mu).
\end{equation}
The subindex $h$ in $\Jhat_h^{\text{loc}}$ refers to the fact that the construction of the solution space $V_{H,\kl,\mu}^{\text{ms}}$ for solving \eqref{eq:PG_LOD} internally requires the computation of the correctors that resolve the fine-scale mesh, which can then be discarded immediately.
Note that we do not plugin $\uhkmsm \in V_{H,\kl,\mu}^{\text{ms}}$ into $\J$ since the basis of $V_{H,\kl,\mu}^{\text{ms}}$ may not be available, which
is the case if the fine-scale correctors can not be stored.
From \cref{thm:pglod_convergence}, we see that the coarse-scale behavior (in the $L^2$-sense) is captured by $\uhkm \in V_H$ and the correctors are only required for the $H^1$-accuracy.
However, in many multiscale applications, the $L^2$-behavior is already sufficient; see also the discussion in \cite{elf}.
To align with this, we employ the following structural assumption on $\J$. 
\begin{assumption}[$\J$ is a coarse functional] \label{asmpt:coarse_J}
	We assume that the objective functional is a coarse functional, measuring in the $L^2$-sense, i.e.~for all $u_H \in V_H$ and $u^\text{f} \in \Vfh$, we have $$\J(u_H + u^\text{f}, \mu) = \J(u_H, \mu).$$ 
\end{assumption}

\textbf{Approximation of the Gradient.}
For our optimization method, we require the gradient of $\Jhat_h^{\text{loc}}$.
As discussed in \cref{sec:problem}, we use the adjoint variable, which we also compute with the PG--LOD.
While \eqref{eq:PG_LOD} works as a replacement for \eqref{P.state}, we formulate a corresponding PG--LOD version of the dual problem for \eqref{eq:dual_solution}:
Seek a function $\phkmsm \in V_{H,\kl,\mu}^{\text{ms}}$ such that
\begin{equation}
\label{eq:PG_LOD_dual}
a_{\mu}(v_H, \phkmsm) =  \partial_u \J(\uhkm, \mu)[v_H] \qquad \text{for all } \; v_H \in V_H.
\end{equation}
Note that \cref{asmpt:coarse_J} justifies that $\uhkm$  (and not $\uhkmsm$) is used for the right-hand side of \eqref{eq:PG_LOD_dual}.
From \eqref{eq:PG_LOD_dual}, we conclude that, just as the FOM in \cite{paper1}, the localized FOM is a conforming choice in the sense that~$\uhkmsm$ and~$\phkmsm$ belong to the same space~$V_{H,\kl,\mu}^{\text{ms}}$.
This choice only makes sense if the given multiscale coefficient $A_\mu$ is symmetric, as we have assumed throughout this article.
In that case, the recaptured multiscale effects for the primal and dual operators are the same.
If instead $A_\mu$ is not symmetric, different LOD spaces must be constructed, which we do not consider.

Finally, we compute the gradient information with the following formula:
\begin{equation} \label{eq:LOD_gradient}
	\nabla_\mu \Jhat_h^{\text{loc}}(\mu) = \partial_\mu \J(\uhkm, \mu) +  \partial_\mu r_{\mu}^\pr(\uhkm)[\phkm].
\end{equation}
We emphasize again that we do not use $\uhkmsm \in V_{H,\kl,\mu}^{\text{ms}}$ and $\phkmsm \in V_{H,\kl,\mu}^{\text{ms}}$ but instead their coarse-scale representations $\uhkm \in V_H$ and $\phkm \in V_H$ to be able to discard corrector information directly after their computation.
Note that we could still plugin $\uhkmsm  \in V_{H,\kl,\mu}^{\text{ms}}$ at some places in \eqref{eq:LOD_gradient}, e.g., for the linear terms of~$\J$ since the related terms can be prepared simultaneously to the assembly of $\Ktmu$.

To keep the theory short, we do not consider Hessian information in the localized approach but note that using Newton's method as in \cite{paper2} is straightforward.

\subsection{Localized reduced-order model using the TSRBLOD}
\label{sec:TRLRB_with_TSRBLOD}

To derive an online efficient reduced-order model for the PG--LOD, we recall the TSRBLOD recently introduced in \cite{tsrblod}.
The {TS\-RBLOD} can be divided into two reduction processes.
In Stage~1, RB models for the corrector problems are constructed, and in Stage~2, these RB correctors are combined to a reduced two-scale formulation to reduce the global LOD scheme to a single reduced model.
The idea of reducing the corrector problems similar to Stage~1 has already been proposed as RBLOD in~\cite{RBLOD}.
While in \cite{tsrblod}, the TSRBLOD showed to be more beneficial in terms of online efficiency, the additional coarse-scale reduction introduces a different approximation error.
However, the additional error of the TSRBLOD can rigorously be controlled. 
As demonstrated in \cite{tsrblod}, for large coarse systems, the online-acceleration can be multiple orders of magnitude.
On the other hand, the offline cost of the TSRBLOD is higher than the RBLOD since an additional offline-online decomposition is to be performed.
To avoid an overload of methods, in this paper, we only consider the TSRBLOD for the relaxed TR method but mention that the same ideas can immediately be transferred to the RBLOD from \cite{RBLOD}.\\[-.5em]

\textbf{Two-scale formulation of the PG--LOD.}
To relate \eqref{eq:PG_LOD} to the two-scale-based view on the PG--LOD as used in~\cite{tsrblod}, we further note that there exists a uniquely defined two-scale representation $\utsm \in \Vts$ of~$\uhkmsm \in V_{H,\kl,\mu}^{\text{ms}}$, in the two-scale space
$$ \Vts := V_H \oplus \Vfhkt[T_1] \oplus \cdots \oplus \Vfhkt[\Tlast].$$
For $\uts = (u_H, \uft[T_1], \dots, \uft[\Tlast]) \in \Vts$ we define the corresponding two-scale $H^1$-norm of~$\uts$ by
\begin{equation*}
	\Snorm{\uts}^2 := \snorm{u_H}^2 + \sumT \snorm*{\uft}^2.
\end{equation*}
The two-scale approximation $\utsm \in \Vts$ is the solution of 
\begin{equation}
	\Btsm\left(\utsm,  \vts \right) = \Fts_\mu(\vts) \qquad \text{for all } \vts \in \mathcal{V}, \label{eq:two_scale_TRLRB}
\end{equation}
where we define the two-scale bilinear form $\Btsm \in \text{Bil}(\Vts)$ given by
\begin{multline*}
	\Btsm\left((u_H, \ufti{1}, \dots, \ufti{{\abs{\mathcal{T}_H}}}),  (v_H, \vfti{1}, \dots, \vfti{{\abs{\mathcal{T}_H}}})\right):=
	a_{\mu}(u_H - \sumT \uft, v_H) + \rho^{1/2}\sumT a_{\mu}(\uft, \vft) - a_{\mu}^T(u_H, \vft),
\end{multline*}
with a stabilization parameter $\rho \geq 1$; cf. \cite{tsrblod}. Further, let $\Fts_\mu \in \Vts^\prime$ be given as
\begin{align*}
	\Fts_\mu\left((v_H, \vfti{1}, \dots, \vfti{{\abs{\mathcal{T}_H}}})\right) &:= l_\mu(v_H).
\end{align*}
As proven in \cite{tsrblod}, the two-scale solution $\utsm \in \Vts$ can always be constructed from~\eqref{eq:PG_LOD} and the respective fine-scale correctors, such that
\begin{equation}\label{eq:two_scale_solution}
	\utsm = \left[\uhkm,\, \QQktm[T_1](\uhkm),\, \ldots,\, \QQktm[\Tlast](\uhkm)\right].
\end{equation}
Similarly, with  
\begin{align} \label{eq:dual_two_scale_rhs}
	\Fts^\du_{\mu,\bullet}\left((v_H, \vfti{1}, \dots, \vfti{{\abs{\mathcal{T}_H}}})\right) := \partial_u \J(\bullet, \mu)[v_H],
\end{align}
we can reformulate the dual system \eqref{eq:PG_LOD_dual} by solving for the two-scale dual solution $\ptsm \in \Vts$ of
\begin{equation}
	\Btsm\left(\ptsm,  \vts \right) = \Fts^\du_{\mu,\uhkm}(\vts) \qquad \text{for all } \vts \in \mathcal{V}, \label{eq:two_scale_dual_TRLRB}
\end{equation}
where we note that we did not flip the arguments due to the symmetry of the bilinear form $a_\mu$.\\

\textbf{Stage 1 of TSRBLOD.}
In Stage~1 of the TSRBLOD reduction process, we construct reduced spaces for the corrector problems \eqref{eq:loc_cor_problems} for each~$T$, parameterized towards the respective FE shape functions on $T$.
Assuming such respective Stage~1 spaces $\Vfrbkt[T]$ to be given, we form a reduced two-scale space
$$\Vtsrblod := V_H \oplus \Vfrbkt[T_1] \oplus \dots \oplus \Vfrbkt[T_{\Tlast}] \subset \Vts$$
which can be used to consider an RBLOD-type version of the two-scale equations \eqref{eq:two_scale_TRLRB} and \eqref{eq:two_scale_dual_TRLRB}, where the respective solutions can be obtained by using the RBLOD.\\

\textbf{Stage 2 of TSRBLOD.}
For an online efficient reduced model, loops over the coarse mesh should be avoided in the online phase. For this reason, in Stage~2 of the two-scale reduction, we construct a reduced basis of $\Vtsrblod$.

Since the primal and dual equations \eqref{eq:two_scale_TRLRB} and \eqref{eq:two_scale_dual_TRLRB} have different right-hand sides, we require two two-scale reduced spaces $\Vtsrbpr, \Vtsrbdu \subset \Vtsrblod$.
Reducing the primal equation \eqref{eq:two_scale_TRLRB}, given $\Vtsrbpr$, means to compute the two-scale reduced primal solution $\utsmrb \in \Vtsrbpr$ by
\begin{equation}\label{eq:tsrblod_primal}
\utsmrb := \argmin_{\uts \in \Vtsrbpr}\, \sup_{\vts \in \Vts} \frac{\Fts_\mu(\vts) - \Btsm(\uts, \vts)}{\Snorm{\vts}}.
\end{equation}
Further, let $\uhkmrb \in V_H$ denote the resulting TSRBLOD coarse-scale approximation, which can be reconstructed from $\utsmrb$, just by using the $V_H$-part of the respective basis of $\Vtsrbpr$.
Then, we define the corresponding reduced functional by
\begin{equation}\label{eq:TSRBLOD_func}
	\Jhat_\red^{\text{loc}}(\mu) := \J(\uhkmrb, \mu).
\end{equation}
Given the dual two-scale reduced space $\Vtsrbdu$, the reduced dual problem \eqref{eq:two_scale_dual_TRLRB} can be defined analogously, with the vital difference that the right-hand side of the Stage~2 FOM system needs to be adjusted with the one from the dual problem \eqref{eq:two_scale_dual_TRLRB}.
By replacing $\Fts_\mu$  in~\eqref{eq:tsrblod_primal} by $\Fts^\du_{\mu,\uhkmrb}$ from~\eqref{eq:dual_two_scale_rhs} and using $\Vtsrbdu$ instead, we obtain the two-scale dual solution $\ptsmrb \in \Vtsrbdu$ by
\begin{equation}\label{eq:tsrblod_dual}
\ptsmrb := \argmin_{\pts \in \Vtsrbdu} \,\sup_{\vts \in \Vts} \frac{\Fts^\du_{\mu,\uhkmrb}(\vts) - \Btsm(\pts, \vts)}{\Snorm{\vts}},
\end{equation}
which again uses the symmetry of $a_\mu$.
With the resulting coarse approximation $\phkmrb \in V_H$ reconstructed from~$\ptsmrb$, we can compute the reduced gradient as
\begin{equation} \label{eq:TSRBLOD_gradient}
	\nabla_\mu \Jhat_\red^{\text{loc}}(\mu) = \partial_\mu \J(\uhkmrb, \mu) +  \partial_\mu r_{\mu}^\pr(\uhkmrb)[\phkmrb].
\end{equation}

\begin{remark}[Generalization of the TSRBLOD approach]\label{rem:TSRBLOD_has_mu_rhs}
	We emphasize that the TSRBLOD approach in \cite{tsrblod} did not consider a parameterized right-hand side or an output functional.
	We still omit a further technical description for brevity, noting that an efficient online system can still be observed.
\end{remark}

\subsection{A posteriori error estimate for the reduced functional}
\label{sec:a_posteriori_estimates}

For the localized FOM and ROM approximation schemes, we aim at deriving the error estimator~$\Delta_{\Jhat_r^{\text{loc}}}$ of the reduced functional, which is needed for characterizing the TR in~\eqref{eq:rel_TRRBsubprob}.
In the sequel, we restrict our theoretical findings to the linear-quadratic case of $\J$.

\begin{assumption} \label{asmpt:J}
	We assume that $\J$ can be decomposed into a parameter function $\Theta \colon \Params \to \mathbb{R}$, and a (parameter dependent) linear and bilinear term $j_\mu: V \to \R$ and $k_\mu: V \times V \to \R$ that are (bi-)linear (and symmetric) for every parameter $\mu \in \Params$, such that
	\begin{equation*}
		\J(u, \mu) = \Theta(\mu) + j_\mu(u) + k_\mu(u, u).
	\end{equation*}
\end{assumption}

In what follows, we transfer the a posteriori result from \cite{tsrblod} to the two-scale formulations of the primal and dual systems.
On top of that, similar to the a posteriori result in \cite{paper1,QGVW2017}, we combine a primal and dual estimate to obtain an estimator for the reduced functional.
To this end, we use the following norms to assess the approximation quality of the two-scale approach:
\begin{align*}
\Anorm{\uts}^2 &:= 
\anorm{u_H - \sumT \uft}^2
+ \rho\sumT \anorm{\QQktm(u_H) - \uft}^2, \quad
\SMnorm{\uts}^2 &:= 
\snorm{u_H}^2
+ \rho\sumT \snorm{\QQktm(u_H) - \uft}^2.
\end{align*}

\begin{proposition}[Upper bound on the local primal model reduction error] \label{prop:primal_rom_error_LOD}
	For $\mu \in \Params$,  let $\utsm \in \Vts$ be the solution of~\eqref{eq:two_scale_TRLRB}.
	Let $\utsmrb \in \Vtsrb$ be the two-scale reduced solution of~\eqref{eq:tsrblod_primal}.
	Then, it holds
	\begin{align}
		\Anorm{\utsm - \utsmrb} \leq \Delta_{\pr}^{\text{rb}}(\mu) := \eta_{a,\mu}(\utsmrb),
	\ \text{ with } \quad
	\label{eq:stage2_est_a}
		\eta^\pr_{a,\mu}(\uts) &:= 
		\sqrt{5}(\CPG)^{-1}
		\sup_{v\in\Vts}
		\frac{\Fts_\mu(\vts) - \Btsm(\uts,\vts)}
		{\Snorm{\vts}}.
	\end{align}
\end{proposition}
\begin{proof}
	The assertion follows directly from \cite{tsrblod}, using the inf-sup stability of the two-scale equation, see~\eqref{eq:infsup_constant}.
\end{proof}

\begin{remark}[Equivalence of the two-scale norms] \label{rmk:norm_equivalence_LOD}
	Due to the definitions of $\Anorm{\cdot}$, $\Snorm{\cdot}$, $\anorm{\cdot}$, and $\snorm{\cdot}$ and the  equivalences of $\Anorm{\cdot}$ and $\Snorm{\cdot}$, as well as $\anorm{\cdot}$, and $\snorm{\cdot}$, respectively, we note that the fine-scale errors $\anorm{\uhkmsm - \uhkmsmrb}$, $\snorm{\uhkmsm - \uhkmsmrb}$, and the coarse-scale errors $\anorm{\uhkm - \uhkmrb}$, and~$\snorm{\uhkm - \uhkmrb}$ can be bounded by $\Delta_{\pr}^{\text{rb}}$ with the respective equivalence constants, cf.~\cite{tsrblod}.
\end{remark}

Next, we derive corresponding dual estimates that account for the fact that the right-hand side contains the reduced primal solution instead of the true LOD solution.

\begin{proposition}[Upper bound on the local dual model reduction error]
\label{prop:dual_rom_error_LOD}
	For $\mu \in \Params$, let $\ptsm \in \Vts$ be the solution of the two-scale dual equation \eqref{eq:two_scale_dual_TRLRB}.
	Let $\ptsmrb \in \Vtsrb$ be the two-scale reduced dual solution of~\eqref{eq:tsrblod_dual}.
	Then, it holds
	\begin{align}
	\Anorm{\ptsm - \ptsmrb} \leq \Delta_{\du}^{\text{rb}}(\mu) := \frac{\sqrt{5}}{\CPG} \left(2 \cont{k_\mu}\Delta_{\pr}^{\text{rb}}(\mu)   + \eta^\du_{a,\mu}(\ptsmrb)\right),
	\end{align}
	where $\cont{k_\mu}$ denotes the continuity constant of $k_\mu$ and $\eta^\du_{a,\mu}$ is analogously defined as $\eta^\pr_{a,\mu}$ associated with \eqref{eq:two_scale_dual_TRLRB}, i.e.
	\begin{equation}
		\eta^\du_{a,\mu}(\pts) := 
	\sqrt{5}(\CPG)^{-1}
	\sup_{v\in\Vts}
	\frac{\Fts^\du_{\mu,\uhkm}(\vts) - \Btsm(\pts,\vts)}
	{\Snorm{\vts}}
	\end{equation}
\end{proposition}
\begin{proof}
	We use the shorthands $\mathfrak{e}_{\mu}^{\du} := \ptsm - \ptsmrb \in \Vtsrb$ and $e_{H, \mu}^{\pr} := \uhkm - \uhkmrb$, where $\uhkm \in V_H$ and~$\uhkmrb$ are the $V_H$ parts of $\utsm$ and $\utsmrb$, respectively.
	With the inf-sup stability constant from \cref{eq:infsup_constant}, we have
	\begin{align*}
	\frac{\CPG}{\sqrt{5}} \, \Anorm{\mathfrak{e}_{\mu}^{\du}} &\leq \sup_{0 \neq \vts \in \Vts}
	\frac{\Btsm(\mathfrak{e}_{\mu}^{\du},\vts)}{\Snorm{\vts}}  = 
	\sup_{0 \neq \vts \in \Vts} \left(
	\frac{\Fts^\du_{\uhkm}(\vts)}{\Snorm{\vts}} - \frac{\Btsm(\ptsmrb,\vts)}{\Snorm{\vts}}\right) \\
	&= 	 \sup_{0 \neq \vts \in \Vts} \left(
	\frac{\Fts^\du_{\uhkm}(\vts)}{\Snorm{\vts}} - \frac{\Fts^\du_{\uhkmrb}(\vts)}{\Snorm{\vts}} + \frac{\Fts^\du_{\uhkmrb}(\vts)}{\Snorm{\vts}} - \frac{\Btsm(\ptsmrb,\vts)}{\Snorm{\vts}}\right) 	
	\leq 2 \|k_\mu\|\; \|e_{H, \mu}^\pr\|\; + \eta^\du_{a,\mu}(\ptsmrb),
	\end{align*}
	with $\Fts^\du_{\mu,\bullet}$ defined in~\eqref{eq:dual_two_scale_rhs}, which is linear in its sub-index argument due to the definition of $\J$.
	In the last inequality we used that $ \partial_u \J(\bullet, \mu)[v_H] = j_\mu(v_H) + 2k_\mu(v_H, \bullet)$ for the linear-quadratic case.
	We attain the desired result utilizing \cref{prop:primal_rom_error_LOD} and \cref{rmk:norm_equivalence_LOD}.
\end{proof}

Similar to \cref{rmk:norm_equivalence_LOD}, the respective dual estimators can also bound the corresponding dual norms from \cref{prop:dual_rom_error_LOD}.
Finally, we derive the a posteriori error result for the reduced objective functional.

\begin{proposition}[Upper bound for the reduced functionals]
	\label{prop:LOD_reduced_functional}
	For $\mu \in \Params$ let $\utsm \in \Vts$ be the two-scale solution of \eqref{eq:two_scale_TRLRB} with coarse part $\uhkm \in V_H$ and LOD-space representation $\uhkmsm \in V_{H,\kl,\mu}^{\text{ms}}$.
	Further, let $\ptsm \in \Vts$ be the two-scale solution of \eqref{eq:two_scale_dual_TRLRB} with coarse part $\phkm \in V_H$ and LOD-space representation $\phkmsm \in V_{H,\kl,\mu}^{\text{ms}}$.
	\begin{enumerate}
		\item [\emph{(i)}] We have for the TSRBLOD reduced cost functional
		\begin{align*}
		\hspace{-0.5cm} |\Jhat_h^{\text{loc}}(\mu) - \Jhat_\red^{\text{loc}}(\mu)| \leq \Delta_{\Jhat_\red^{\text{loc}}}(\mu) :=  \Delta_{\pr}^{\text{rb}}(\mu) &\eta^\du_{a,\mu}(\ptsmrb) 
		+ (\Delta_{\pr}^{\text{rb}}(\mu))^2 \cont{\kformd}  + \Delta_{\text{trunc}}^{\text{rb}}(\mu),
		\end{align*}
		where $\ptsmrb \in \Vtsrbdu$ denotes the two-scale reduced dual equation and
		$\Delta_{\text{trunc}}^{\text{rb}}(\mu)$ is a truncation-reduction-based homogenization term which is specified below
		\item[(ii)] The truncation-reduction-based homogenization term $\Delta_{\text{trunc}}^{\text{rb}}(\mu)$ is defined as 
		\begin{equation} \label{eq:hom_red_term}
		 \Delta_{\text{trunc}}^{\text{rb}}(\mu) := a_\mu(\ehms,\QQkmrb(\phkmrb)) 
		\end{equation}
		and can be estimated by
		\begin{equation} \label{eq:estimation_hom_red_term}
			\hspace{-0.5cm}\abs{\Delta_{\text{trunc}}^{\text{rb}}(\mu)} \leq \Delta_{\pr}^{\text{rb}}(\mu) \left(c \, \kl^{d/2} \theta^\ell \snorm{\phkmrb} + \eta^\du_{a, \mu}(\ptsmrb)\right) + \alpha^{-1/2}\snorm{\phkmrb} \eta^\pr_{a, \mu}(\utsmrb),
		\end{equation}
		with respective coarse- and two-scale-space primal and dual solutions and constant $c>0$.
	\end{enumerate}
\end{proposition}
\begin{proof}
	We utilize \cref{asmpt:coarse_J} to incorporate the estimates of \cref{prop:primal_rom_error_LOD} and \cref{prop:dual_rom_error_LOD}.
	By using the shorthands $\ehms := \uhkmsm - \uhkmsmrb$ and $\eh := \uhkm - \uhkmrb$ and the definition of $r_{\mu}^\du$ in \eqref{eq:dual_residual} with~$ \partial_u \J(\bullet, \mu)[q] = j_\mu(q) + 2k_\mu(q, \bullet)$, we have 
	\begin{align*}
		|\Jhat_h^{\text{loc}}&(\mu) - \Jhat_\red^{\text{loc}}(\mu)| = |\J(\uhkm, \mu) - \J(\uhkmrb, \mu)| \\ 
		&= |j_\mu(\ehms) + \kformd(\uhkm, \uhkm) - \kformd(\uhkmrb, \uhkmrb)
		- a_\mu(\ehms,\phkmsmrb) + a_\mu(\ehms,\phkmsmrb)| \\
		&= |r_{\mu}^\du(\uhkmrb, \phkmsmrb)[\ehms] - 2\kformd(\uhkmrb, \ehms) + \kformd(\uhkm, \uhkm) - \kformd(\uhkmrb, \uhkmrb) + a_\mu(\ehms,\phkmsmrb)| \\
		&= |r_{\mu}^\du(\uhkmrb, \phkmsmrb)[\ehms] + \kformd(\ehms, \ehms) + a_\mu(\ehms,\phkmsmrb)| \\
		&\leq \eta^\du_{a,\mu}(\ptsmrb)\; \|\ehms\| +\cont{\kformd} \; \|\ehms\|^2 
		+ |a_\mu(\ehms,\QQkmrb(\phkmrb))|,
	\end{align*}
	where we used that
	$a_\mu(\ehms,\phkmsmrb) = - a_\mu(\ehms,\QQkmrb(\phkmrb))$.
	This concludes the proof for (i). For (ii), we note that	
	\begin{align*}
		a_\mu(\ehms,\QQkmrb(\phkmrb))  &= a_\mu(\ehms,\QQm(\phkmrb))  \mkern-3mu 
		- a_\mu(\ehms,(\QQm \mkern-3mu - \mkern-3mu \QQkm \mkern-3mu + \mkern-3mu\QQkm \mkern-3mu - \mkern-3mu \QQkmrb)(\phkmrb))  
	\end{align*}
	and
	\begin{align*}
		a_\mu(&\ehms,\QQm(\phkmrb))
		= a_\mu((\QQkm \mkern-3mu - \mkern-3mu \QQm)(\uhkm - \uhkmrb) + (\QQkm \mkern-3mu - \mkern-3mu \QQkmrb)(\uhkmrb),\QQm(\phkmrb)).
	\end{align*}
	We thus obtain
	\begin{align*}
		|a_\mu(\ehms,\QQkmrb(\phkmrb))| &\leq \anorm{\ehms} (\anorm{(\QQm \mkern-3mu - \mkern-3mu \QQkm)(\phkmrb)}
		+ \anorm{(\QQkm \mkern-3mu - \mkern-3mu \QQkmrb)(\phkmrb)}) \\
		& \quad + \anorm{\QQm(\phkmrb)} (\anorm{(\QQm \mkern-3mu - \mkern-3mu \QQkm)(\eh)} +
		\anorm{(\QQkm \mkern-3mu - \mkern-3mu \QQkmrb)(\uhkmrb)}) \\
		& \leq \Delta_{\pr}^{\text{rb}}(\mu) ( c_1\ell^{d/2} \theta^\ell \mkern-1mu \anorm{\phkmrb}
		+\mkern-3mu \eta^\du_{a, \mu}(\ptsmrb)) \mkern-3mu
		 +\mkern-3mu \anorm{\phkmrb}\mkern-1mu ( c_2 \ell^{d/2} \theta^\ell \anorm{\eh} \mkern-3mu+ \mkern-3mu\eta^\pr_{a, \mu}(\utsmrb)) \\
		& \leq \Delta_{\pr}^{\text{rb}}(\mu) \left((c_1 + c_2) \ell^{d/2} \theta^\ell \anorm{\phkmrb}
		+ \eta^\du_{a, \mu}(\ptsmrb)\right) + \anorm{\phkmrb} \eta^\pr_{a, \mu}(\utsmrb),
	\end{align*}
	where we have used the a priori result on the corrector decay \eqref{eq:exponential_decay} and \cref{rmk:norm_equivalence_LOD}. Using the equivalence of $\anorm{\cdot}$ and $\snorm{\cdot}$ yields the assertion.
\end{proof}

\begin{remark}[Truncation-reduction-based homogenization term] \label{rem:trunc-red-hom-term}
	In \cref{prop:LOD_reduced_functional}, we intentionally separated the error estimation from the homogenization term $\Delta_{\text{trunc}}^{rb}(\mu)$ and presented a rather naive estimation of it.
	The reason is that the term can be interpreted as a truncation term that (without reduction) vanishes for  true LOD-space functions, i.e.
	\begin{equation}
		a_\mu(v_{H,\mu}^{\text{ms}},\QQm(\phkmrb)) = 0,
	\end{equation}
	for all $v_{H,\mu}^{\text{ms}} \in V_{H,\mu}^{\text{ms}}$, since $\QQm(\phkmrb) \in \Vfh$ and $V_h = V_{H,\mu}^{\text{ms}} \oplus_{a_{\mu}} \Vfh$.
\end{remark}

The computation of the above-derived estimators can be offline-online decomposed with a numerically stable procedure, see~\cite{tsrblod} for the primal equation.
However, while the additional orthonormalization of the residual terms of Stage~1 is necessary for the Stage~2 residual, the additional expenses for stabilizing the Stage 2 residual are not strictly needed in our approach.
Indeed, concerning the overall cost of the TR-TSRBLOD algorithm, we omit the offline-online decomposition of Stage 2 entirely and instead compute the residual and its Riesz-representative whenever needed, cf.~\cref{sec:pseudo-code}.

\subsection{Local basis enrichment}
\label{sec:local_basis_enrichment}

It remains to elaborate on the adaptive localized enrichment strategy for a parameter $\mu \in \Params$, e.g., an accepted outer iterate of the TR algorithm.
In \cite{paper2}, it is discussed that the RB space can either be updated unconditionally or optionally.
For local RB models, the situation is more complex.
While we, at least for obtaining certified convergence in the sense of \cref{Thm:convergence_of_TR}, always perform an enrichment, some local models may reject or dismiss the snapshots if, e.g., the selected parameter does not influence the local model.
For this reason, a localized error that decides for a local update, known as localized online enrichment; cf.~\cite{Buh2017}.

First of all, we note that the estimators $\eta_{a, \mu}^\pr$ and $\eta_{a, \mu}^\du$ that occur in \cref{prop:LOD_reduced_functional} can indeed be boiled down to their respective local reduction errors by the construction of the two-scale bilinear form. In particular, the standard RB estimation of Stage~1 of the reduction process for the TSRBLOD can be used. 

For each $T \in \Gridh$, we may use the residual-norm based estimate
\begin{equation}\label{eq:stage1_error_estimate}
	\anorm{\QQktm(v_H) - \QQktmrb(v_H)} 
	\leq \eta_{T,\mu}(\QQktmrb(v_H)),
\end{equation}
where
\begin{equation} \label{eq:stage_1_estimator}
	\eta_{T,\mu}(\QQktmrb(v_H)) := \alpha^{-1/2} \sup_{\vft\in\Vfhkt}
	\frac{a_\mu^T(v_H, \vft) - a_\mu(\QQktmrb(v_H), \vft)}{\snorm{\vft}},
\end{equation}
which is essentially the standard residual-based estimation of \eqref{eq:loc_cor_problems}.

In the relaxed TR scheme, at an enrichment step for a new parameter $\mu$, for every $T$, we use the Stage~1 estimator $\Delta^T_{\text{loc}}(\mu) := \eta_{T,\mu}(\QQktmrb(v_H))$ to decide for whether we enrich the local space.
We enrich the space if the estimator is larger than a tolerance $\tau_{{\text{loc}}} > 0$, relaxed for every outer iteration $k$, i.e.
\begin{equation}\label{eq:optional_enrichment}
	\Delta^T_{\text{loc}}(\mu) + \varepsilon^{(k)} > \tau_{{\text{loc}}}.
\end{equation}
For a sufficiently small $\tau_{{\text{loc}}}$ or a large relaxation $\varepsilon^{(k)}$, the enrichment strategy can be considered unconditionally.
We note that the presence of the relaxation parameter in \eqref{eq:optional_enrichment} can be justified by the fact that the outer iteration steps of the relaxed TR can be expected to be far away from each other.

We note that the online adaptive approach is also motivated by the numerical experiments in~\cite{tsrblod,RBLOD},
where it was demonstrated that moderate choices of $\tau_{{\text{loc}}}$ already produce acceptable reduced models.
However, it is clear that the choice of the tolerance $\tau_{{\text{loc}}}$ is highly problem dependent.
If the tolerance is chosen too large, the method could be stagnant (due to the missing local basis quality).
In such cases, it is recommended to refine the tolerance adaptively.
For simplicity, we omit such a strategy in this paper.

Concerning the Stage~2 reduction, we note that the TS\-RB\-LOD model is based on the reduced models from Stage~1, meaning that whenever the Stage~1 models are enriched, it is recommended to build the new Stage~2 from scratch.
Thus, there is more freedom in choosing the enrichment parameters for the TSRBLOD model.
With respect to the fact that, at iteration $k$, the Stage~1 models are exact (up to the tolerance $\tau_{{\text{loc}}}$), we propose to enrich the TSRBLOD model for the same sequence of TR iterates $\mu^{(i)}$, for $i = 0, \dots, k$.
Greedy-based enrichments of the TSRBLOD are also possible, mainly because the snapshot generation with Stage~1 is fast.
However, our experiments suggested that greedy-search algorithms do not provide significantly different results.

\subsection{TR-TSRBLOD algorithm in Pseudo-code} \label{sec:pseudo-code}
We summarize the (R)-TR algorithms based on the TSRBLOD in the following.
The relaxed TR-TSRBLOD procedure in \cref{alg:R-TR-TSRBLOD} is analog to \cref{alg:basic_TRRB} but with the specification of a localized LOD-based FOM and the localized TSRBLOD reduced model including its respective estimator is used.
Thus, no FEM-based approximations are required compared to the algorithm used in \cite{paper1}.
This makes the TR algorithm usable for a much more comprehensive range of optimization problems.
As stated in Line \ref{stop_early_rel} of \cref{alg:R-TR-TSRBLOD}, we check the FOM termination criterion prior to the enrichment.
This is because the online enrichment, including the assembly of the respective estimators, is relatively more expensive than the pure computation of the termination criterion.
Again, if $\varepsilon^{(k)} \equiv 0$, the relaxed algorithm is equivalent to the original TR algorithm.

\begin{algorithm2e}[!h]
	\KwData{Initial parameter $\mu^{(0)}$, stopping tolerance for the sub-problem $\tau_\textnormal{{sub}}\ll 1$, stopping tolerance for the FOC condition $\tau_\textnormal{{FOC}}$ with $\tau_\textnormal{{sub}}\leq\tau_\textnormal{{FOC}}\ll 1$, relaxation sequence $(\varepsilon^{(k)})_k$.}
	Set $k=0$ and initialize TSRBLOD model with $\mu^{(0)}$\; 
	\While
	{$\|\mu^{(k)}-\Proj_{\Paramsad}(\mu^{(k)}-\nabla_\mu \Jhat^{\textnormal{loc}}_h(\mu^{(k)}))\|_2 > \tau_{\textnormal{{FOC}}}$\label{stopping_condition_R}}{
		Compute $\mu^{(k+1)}$ from~\eqref{eq:rel_TRRBsubprob} with relaxed termination~\eqref{FOC_sub-problem} and \eqref{eq:rel_cut_of_TR}\label{solve_sub_problem_R}\;
		\uIf{Relaxed sufficient decrease condition~\eqref{eq:suff_decrease_condition_rel} is fulfilled with relaxation $\varepsilon^{(k)}$} {
			Accept $\mu^{(k+1)}$ and possibly enlarge the TR-radius\;
			Before TSRBLOD enrichment: If \eqref{eq:global_termination_TR} go to Line 2 for early termination\label{stop_early_rel}\;
			\textbf{Stage 1:} Enrich the local RB corrector models at $\mu^{(k+1)}$ for all $T$ that fulfill \eqref{eq:optional_enrichment}\;
			\textbf{Stage 2:} Construct the primal and dual two-scale models and enrich for all $\mu^{(k')}$, $k'=0,\dots,k+1$, and do not assemble Stage 2 estimator\;
		}
		\Else {
			Reject $\mu^{(k+1)}$, shrink the TR radius $\delta^{(k+1)} = \beta_1\delta^{(k)}$ and go to \ref{solve_sub_problem}\;		
		}
		Set $k=k+1$\;
	}
	\caption{Relaxed TR-TSRBLOD algorithm}
	\label{alg:R-TR-TSRBLOD}
\end{algorithm2e}

\section{Numerical experiments}
\label{sec:loc_experiments}

We analyze the presented (relaxed) TR-TSRBLOD approach with two experiments with the same problem description, only differing in their respective multiscale complexity.
We define the fine-mesh by $n_h \times n_h$ and the coarse-mesh by $n_H \times n_H$ quadrilateral grid-blocks of $\Omega:=[0,1]^2$, used to determine the standard FE mesh~$\mathcal{T}_h$ and $\mathcal{T}_H$, respectively, with traditional $\mathcal{P}^1$-FE spaces $V_h$ and $V_H$.
The mesh sizes $h$ and $H$ can be computed from $n_h$ and $n_H$.
In the first small experiment, we compare the localized methods to FEM-based TR methods. In the second large experiment, we neglect FEM entirely, as it is computationally infeasible.

In the following, we mainly focus on the number of evaluations relative to the complexity of the fine mesh~$\gridh$, the coarse LOD mesh-size~$H$, or the respective low RB dimensions of the reduced models.
Moreover, we provide run time comparisons that present the computational efficiency observed with our implementation.

Our computations were performed on an HPC cluster with $400$ parallel processes.
Nevertheless, the observed run times can not be interpreted as the minimal computational times of the localized algorithms.
More HPC-oriented implementations can strengthen the localized approaches even more.
We also note that the Stage~2 reduction has been implemented as a serialized process, where
neither the observed data from Stage~1 is efficiently stored, nor the sparsity pattern of the two-scale system matrix is entirely exploited.

We use the $L^2$-misfit objective functional with a Tikhonov-regularization term:
\begin{equation} \label{eq:L2_misfit}
	\J(v, \mu) = \frac{\sigma_d}{2} \int_{D}^{} (I_H(v) - u^{\text{d}})^2 \dx + \frac{1}{2} \sum^{P}_{i=1} \sigma_i (\mu_i-\mu^{\text{d}}_i)^2 + 1.
\end{equation} 
Here, $\mu^{\text{d}} \in \Params$ is the desired parameter and $u^{\text{d}} = I_H(u_{\mu^{\text{d}}})$ the corresponding desired solution specified in each experiment.
Using the interpolation operator $I_H$ ensures \cref{asmpt:coarse_J}.
Moreover, using an actual solution as desired temperature and the respective desired parameter in the objective functional ensures that the optimization problem is sufficiently regular, such that all optimization methods converge to the same point for comparison purposes.
We note that $\J$ can be written in the linear-quadratic form as in \cref{asmpt:J} by $\Theta(\mu) = \frac{1}{2} \sum^{M}_{i=1} \sigma_i (\mu_i-\mu^{\text{d}}_i)^2 + \frac{\sigma_D}{2} \int_{D}^{} u^{\text{d}} u^{\text{d}} + 1$,
$j_{\mu}(u) = -\sigma_D \int_{D}^{} u^{\text{d}}u$,
and
$k_{\mu}(u,v) = \frac{\sigma_D}{2} \int_{D}^{} uv$.

We consider the admissible parameter set $\Params = [1,4]^{24} \times [1,1.2]^8$.
The diffusion coefficient~$A_\mu$ in the symmetric bilinear form $a_\mu$ is considered a $4 \times 4$\ -\,thermal block with two different thermal block multiscale coefficients $A_\mu^1$ and $A_\mu^2$, i.e.
$$
 	A_\mu := A_\mu^1 + A_\mu^2 := \sum_{\xi = 1}^{16} \mu_\xi A_\xi^{1} +  \sum_{\xi =17}^{32} \mu_\xi A_{\xi-16}^{2}.
$$
Each of the $4 \times 4$ blocks is linearly dependent on an individual parameter.
The respective parameterized multiscale blocks are given by $A_\xi^{1} = A^{1}\big|_{\Omega_{i,j}}$ and $A_\xi^{2} = A^{2}\big|_{\Omega_{i,j}}$, where $\Omega_{i,j}$ denotes the $(i,j)$-th thermal block for $i,j=1,2,3,4$ enumerated by $\xi=1,\dots,16$.
The multiscale features are randomly constructed iid values in a normal distribution $\mathcal{N}([0.9,1.1])$ on a $N_1 \times N_1$ (for $A^{1}$) and $N_2 \times N_2$ (for $A^{2}$) quadrilateral grid.
The specific values for $N_1$ and $N_2$ are given for each experiment.
Hence, the multiscale data does not admit periodicity or other structural assumptions apart from the bounds; see \cref{fig:thermal_block} for a visualization of the random field (evaluated with $\mu^{\text{d}}$).
Moreover, both coefficients $A_\mu^{1}$ and $A_\mu^{2}$ have low-conductivity blocks in the middle of the domain, i.e., for $\Omega_{i,j}$, $i,j=2,3$.
The low conductivity is enforced by the choice of the parameter space~$\Params$.
We choose the non-parameterized constant function $f_\mu \equiv 10$ as the right-hand-side function.
For the inner product of $V_h$, we use the energy norm $\norm{\cdot} := \norm{\cdot}_{a,\check{\mu}}$ for a fixed parameter $\check{\mu} \in \Params$ in the middle of the parameter space. Thus, constants in the estimators can be deduced by the min/max-theta approach, cf.~\cite{paper1}.
The maximum contrast $\kappa$ and the respective constants $\alpha$ and $\beta$ can be approximated accordingly.
For the a priori constants in \eqref{eq:estimation_hom_red_term}, we enforced the (in our experiments meaningful) assumption that the dual reduction term dominates the estimate, s.t. we have $c \, \kl^{d/2} \theta^\ell \snorm{\phkmrb} < \eta^\du_{a, \mu}(\ptsmrb)$.
This can be justified by the exponentially decaying term $\theta^{\ell}$ dominating the term for large enough $\ell$.
Furthermore, the algorithm is robust concerning overestimation.

\begin{figure}
	\begin{subfigure}[b]{0.45\textwidth}
		\includegraphics[width=\textwidth]{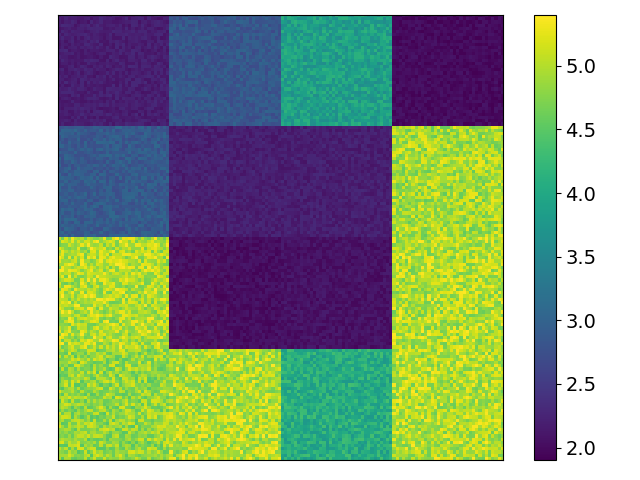}
	\end{subfigure}
	\begin{subfigure}[b]{0.45\textwidth}
		\includegraphics[width=\textwidth]{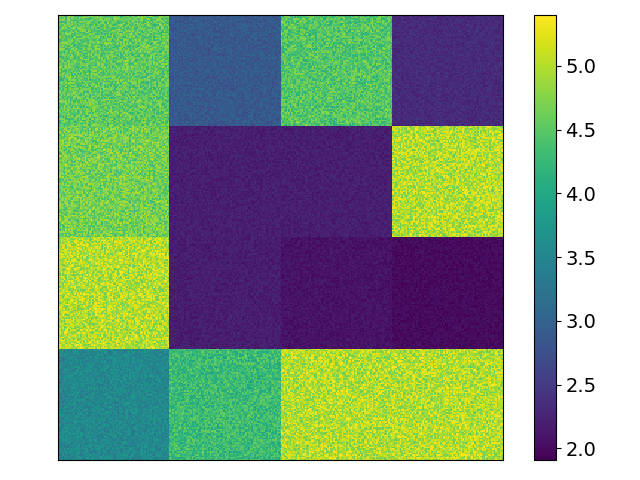}
	\end{subfigure}
	\centering
	\caption{{Coefficient $A^1_\mu$ with $N_1 = 150$ (left) and $A^2_\mu$ with $N_2 = 300$ (right) for the desired state of $\mu^{\text{d}} \in \Params$. For more details, we refer to the accompanying code \cite{CodeTRLRB}.}}
	\label{fig:thermal_block}
\end{figure}

The desired parameter~$\mu^{\text{d}} \in \Params$ is equal for both experiments and mimics the case where boundary constraints are active, i.e., we set $\mu^{\text{d}}_i = 4$ for $i=3,4,6,7,8,9,11,14$ and $\mu^{\text{d}}_i = 1.2$ for $i=28,29,30,31$.
The remaining values of $\mu^{\text{d}}$ are chosen randomly in the interior of $\Params$, see \cref{fig:thermal_block}.
The initial guess $\mu^{(0)}$ is also chosen randomly in the interior of~$\Params$. 
Furthermore, the weights for the objective functional are chosen as~$\sigma_d=100$ and $\sigma_i = 0.001$ for each $i = 1, \dots, P$.

Similar to the experiments in \cite{paper1}, we choose an initial TR radius of $\delta^0 = 0.1$,
a TR shrinking factor~$\beta_1=0.5$, an Armijo step-length $\kappa=0.5$,
a truncation of the TR boundary of $\beta_2 = 0.95$,
a tolerance for enlarging the TR radius of $\eta_\varrho = 0.75$,
a stopping tolerance for the TR sub-problems of $\tau_\text{{sub}} = 10^{-8}$,
a maximum number of TR iteration $K_{\text{max}} = 40$,
a maximum number of sub-problem iterations $K_{\text{{sub}}}= 400$,
a maximum number of Armijo iterations of $50$, and
a stopping tolerance for the FOC condition $\tau_\text{{FOC}}= 10^{-6}$.\\

\textbf{State-of-the-art methods.}
The following algorithms are used to compare to the literature.
\setlist[description]{font=\normalfont\bfseries\space,labelindent=10pt}
\begin{description}
	\item[1. FEM BFGS] Similar to the work in~\cite{QGVW2017,paper1}, we perform a standard projected BFGS method solely based on classical FEM evaluations. This means that the high-fidelity space $V_h$ is the only approximation space, and no reduced approach is used.
	\item[2. TR-RB BFGS from \cite{paper1}] Given the FEM discretization from Method 1, we use the trust-region reduced basis algorithm with full certification and global RB evaluations based on FEM enrichments (\cref{alg:basic_TRRB} with $\varepsilon^{(k)} \equiv 0$).
	As the reduced model, we choose the non-conforming dual (NCD)-corrected approach with Lagrangian enrichment and the respective error estimation.
	This reduction approach has shown very good robustness properties and fast convergence.
	Moreover, we use the projected BFGS as the ROM-based TR sub-problem.
\end{description}

\textbf{Selected methods introduced in this article.}
In this paper, we introduce a relaxation of the TR method. The relaxed version of Method~2 is obtained by choosing the relaxation sequence accordingly.
\begin{description}
	\item[2.r\hspace{7pt}R-TR-RB BFGS] In this method, we use the relaxed trust-region reduced basis variant for the FEM-based TR-RB algorithm from Method 2.
	The relaxation sequence is chosen as $\varepsilon^{(k)} := \mathbbm{1}_{(k < K)}(k) \cdot 10^{10-k}$ where $K := 10$; cf.~\cref{alg:basic_TRRB}.
	\item[2.r$\star$\hspace{2pt}R-TR-RB BFGS] Equivalent to Method 2.r, but without evaluating and assembling the estimator in early iterations~$k < 10$ (because the relaxation dominates the TR-conditions, i.e. $\Delta_{\Jhat^{(k)}_\red} \ll \varepsilon^{(k)}$).
	We use this variant to show that the error estimation is a substantial slow-down factor in the relaxed TR that is not needed in our experiments.
	The method can also be interpreted as Method 2.r with $\varepsilon^{(k)} := \mathbbm{1}_{(k < K)}(k) \cdot a_\text{max}$, where $a_\text{max}$ denotes the maximum positive value in the binary system.
\end{description}
Given the LOD discretization as explained in \cref{sec:tsrblod}, we consider the following localized methods.
\begin{description}
	\item[3. \hspace{6pt}PG--LOD BFGS] As a FEM replacement, we consider a new FOM method using the standard projected BFGS method with PG--LOD evaluations without using reduced models.
	The PG--LOD system is always constructed from scratch and does not use any prior knowledge from previous parameters.
	\item[4. \hspace{6pt}TR-TSRBLOD BFGS] We use the TR algorithm based on the localized TSRBLOD reduction process as detailed in \cref{sec:TRLRB_with_TSRBLOD}.
	The procedure is summarized in \cref{alg:R-TR-TSRBLOD}, choosing $\varepsilon^{(k)} \equiv 0$.
	The sub-problems are again solved with the BFGS method.
	Moroever, we use a local enrichment tolerance $\tau_{{\text{loc}}}=10^{-3}$ in \eqref{eq:optional_enrichment} which has proven to be sufficient for our experiment, cf.~\cref{sec:local_basis_enrichment}.
	\item[4.r\hspace{7pt}R-TR-TSRBLOD BFGS] Just as explained in Method~2.r, we devise the relaxed version of Method~4., by choosing $\varepsilon^{(k)} := \mathbbm{1}_{(k < K)}(k) \cdot 10^{10-k}$ where $K := 10$; cf.~\cref{alg:R-TR-TSRBLOD}.
	\item[4.r$\star$\hspace{2pt}R-TR-TSRBLOD BFGS] Equivalent to Method 4.r, but without evaluating and assembling the estimator if $k < 10$; cf.~Method 2.r$\star$.
	
\end{description}

\textbf{Complexity measures.}
To assess the presented methods w.r.t. their computational demands, we count the accumulated evaluations of the FOM and ROM systems that were needed until the respective algorithm is aborted.
To be precise, we deviate between the following complexities:
\begin{description}
	\item[FEM] FEM evaluations, proportional to the DoFs in $\mathcal{T}_h$, which are needed for approximating \eqref{P.state} or~\eqref{eq:dual_solution} with FEM or for enriching the respective global RB model for Methods 2.r and 2.
	\item[RB] Global RB evaluations for approximating the FEM system, proportional to the global basis size. 
	\item[LOD coarse] Coarse PG--LOD system evaluations with exact corrector data, meaning to solve \eqref{eq:PG_LOD} or~\eqref{eq:PG_LOD_dual}, proportional to the DoFs in the coarse mesh $\mathcal{T}_H$.
	These are only required in Method 3 and for the FOM-based termination criterion in Methods 4.r and 4.
	\item[LOD local] Local evaluations of all FOM corrector problems that are required for assembling the multiscale stiffness matrix of \eqref{eq:PG_LOD} and \eqref{eq:PG_LOD_dual}, locally proportional to the fine DoFs in the coarse-scale patch~$U_\kl(T)$, indicated by the subscript $h$, i.e., $U_\kl(T)_h$.
	\item[RBLOD coarse] Coarse PG--LOD system evaluations with RB-based correctors for the multiscale stiffness matrix, required for the snapshots generation in the TSRBLOD, proportional to the DoFs in~$\mathcal{T}_H$.
	\item[RBLOD local] RB evaluations of the RB corrector problems, proportional to the local RB sizes.
	\item[TSRBLOD] RB evaluations of the TSRBLOD system, proportional to the two-scale RB size.
\end{description}

\textbf{Error measures.}
As the optimization target, we validate the methods by considering the relative error in the optimal value of $\Jhat$, i.e., we consider
$
e^{\Jhat,\text{rel}}(\mu) := |\Jhat(\mu^{\text{d}})  - \Jhat(\mu)| / \Jhat(\mu^{\text{d}}),
$
where~$\mu$ is the current iterate and~$\Jhat$ is either the FEM-based objective functional~$\Jhat_h$ or the LOD-based objective functional~$\Jhat^\textnormal{loc}_h$.

\subsection{Experiment 1: Moderately sized experiment for comparing with FEM-based methods}
\label{sec:exp_1}
In what follows, we consider an experiment where FEM solves are computationally affordable.
To this end, we set the resolution of the multiscale coefficients to $N_1 = 150$ and $N_2 = 300$.
For the fine mesh, we thus choose $n_h = 1200$ to ensure at least $4$ quadrilateral grid cells in each of the rapidly varying multiscale features.
Therefore, the FEM mesh has $1.4$ Mio degrees of freedom.
For the coarse grid, we choose $n_H=20$, which results in only $400$ coarse grid cells and, in particular, $\kl=3$ and $176.400$ fine-mesh elements for full patches~$U_\kl(T)$.
Concerning, the objective functional, we compute $u_{\mu^{\text{d}}}$ as the FEM solution of \eqref{P.state} for $\mu^{\text{d}}$.

\subsubsection{Estimator study for the two-scale reduced functional}
Before we elaborate on the optimization methods, we investigate the above-derived estimator $\Delta_{\Jhat_r^{\text{loc}}}$.
For this purpose, we employ a standard goal oriented greedy-search algorithm. To be precise, we consider a training set~$\Params_{\text{train}}$ containing $100$ randomly sampled parameters.
Subsequently, we enrich all local bases with respect the parameter with the largest estimated error.
In \cref{fig:estimator}, we illustrate the respective largest value of the estimator, compared to its true error and the resulting effectivity.
We conclude that the estimator suffers overestimation but the effectivity stays on a constant level.
This behavior has already been observed in the global RB case in \cite{paper1} and, as explained above, does not harm the method severely.

\begin{figure}[ht]
	\centering%
	\footnotesize%
\begin{tikzpicture}

\definecolor{color0}{rgb}{0.65,0,0.15}
\definecolor{color1}{rgb}{0.84,0.19,0.15}
\definecolor{color2}{rgb}{0.96,0.43,0.26}
\definecolor{color3}{rgb}{0.99,0.68,0.38}
\definecolor{color4}{rgb}{1,0.88,0.56}
\definecolor{color5}{rgb}{0.67,0.85,0.91}
\definecolor{color6}{rgb}{0.27,0.46,0.71}
\definecolor{color7}{rgb}{0.19,0.21,0.58}

\begin{axis}[
  name=top_left,
width=6.5cm,
height=4.35cm,
legend cell align={left},
legend style={fill opacity=0.8, draw opacity=1, text opacity=1, at={(1.2,0)}, anchor=south, draw=white!80!black},
log basis y={10},
ytick pos=right,
x grid style={white!69.0196078431373!black},
xlabel={greedy extension step},
xmajorgrids,
xmin=0, xmax=31,
xtick style={color=black},
y grid style={white!69.0196078431373!black},
ymajorgrids,
ymin=1e-04, ymax=1000,
ymode=log,
ytick style={color=black}
]
\addplot [semithick, color1, mark=*, mark size=2, mark options={solid, rotate=180, fill opacity=0.5}]
table {%
	1 0.009154464469405266
	3 0.0076231735426446345
	5 0.005531703565792245
	7 0.0033016131982213093
	9 0.002942029084725073
	11 0.0012396644590124684
	13 0.0006178145770969756
	15 0.0006554574114185918
	17 0.0006204424012627996
	19 0.00047737509422440816
	21 0.00045824887978662687
	23 0.0004367483127256655
	25 0.00024040758898258296
	27 0.00024071609187936716
	29 0.0002420100872533304
};
\addlegendentry{error J-corr}
\addplot [semithick, color2, mark=triangle*, mark size=2, dashed, mark options={solid, rotate=180, fill opacity=0.5}]
table {%
1 1.3706535601465069
3 0.840500781493063
5 0.6855725889998915
7 0.45380743828746894
9 0.3968604796776081
11 0.20352060647344047
13 0.06581885990666024
15 0.04938253426463875
17 0.03579677297622989
19 0.02520027058537395
21 0.023820216669079308
23 0.02051766163486703
25 0.016272972798919736
27 0.014422760833126552
29 0.012811786786027306
};
\addlegendentry{estimator J corr}
\addplot [semithick, color6, mark=pentagon*, mark size=2, mark options={solid, rotate=90, fill opacity=0.5}]
table {%
	1 181.14141579599112
	3 273.61591480565164
	5 150.72995184718195
	7 160.46482339539824
	9 258.935616974706
	11 458.41332709236843
	13 576.2499933945026
	15 243.52992020086637
	17 146.314386299825
	19 542.3012844463465
	21 489.36279020858433
	23 46.978227590211155
	25 118.3664011141037
	27 192.70392714547754
	29 71.14386118503143
};
\addlegendentry{estimator J eff}
\legend{};
\end{axis}

\begin{customlegend}[legend cell align={left}, legend style={fill opacity=0.8, draw opacity=1, text opacity=1,
    at=(top_left.north west),
    anchor=north east,
    xshift=-5pt,
    inner sep=5pt,
draw=white!80!black},
	legend entries={
    $|\Jhat_h^{\text{loc}} - \Jhat_\red^{\text{loc}}|$,
    $\Delta_{\Jhat_\red^{\text{loc}}}$, 
    $\Delta_{\Jhat_\red^{\text{loc}}}$ eff., 
  }]
\addlegendimage{semithick, color1, mark=*, mark size=2, mark options={solid, rotate=180, fill opacity=0.5}}
\addlegendimage{semithick, color2, mark=triangle*, mark size=2, dashed, mark options={solid, rotate=180, fill opacity=0.5}}
\addlegendimage{semithick, color6, mark=pentagon*, mark size=2, mark options={solid, rotate=90, fill opacity=0.5}}
\end{customlegend}

\end{tikzpicture}
	\caption{%
		Evolution of the true and estimated model reduction error in the reduced functional and its approximations and the effectivities during the greedy basis generation.
		Depicted is the $L^\infty(\Params_\textnormal{train})$-error, i.e.~$|\hat{J}_h - \Jnoncor_\red|$ corresponds to $\max_{\mu \in \Params_\textnormal{train}} |\hat{J}_h(\mu) - \Jnoncor_\red(\mu)|$, $\Delta_{\cJhatn}$ corresponds to $\max_{\mu \in \Params_\textnormal{train}}\Delta_{\cJhatn}(\mu)$, and "$\Delta_{\hat{J}_\red}$ eff." corresponds to $\max_{\mu \in \Params_\textnormal{train}} \Delta_{\hat{J}_\red}(\mu) \,/\, |\hat{J}_h(\mu) - \Jnoncor_\red(\mu)|$.
	}
	\label{fig:estimator}
\end{figure}
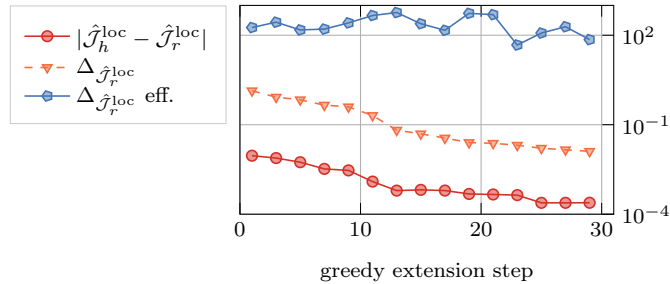

\subsubsection{Comparison of the Optimization methods}

\begin{figure}
	\centering
	\begin{subfigure}{.48\textwidth}
		\centering
		\includegraphics[width=.8\linewidth]{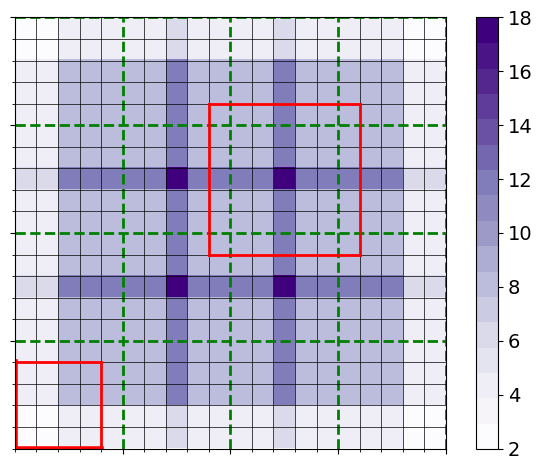}  
	\end{subfigure}
	\begin{subfigure}{.48\textwidth}
		\centering
		\includegraphics[width=.8\linewidth]{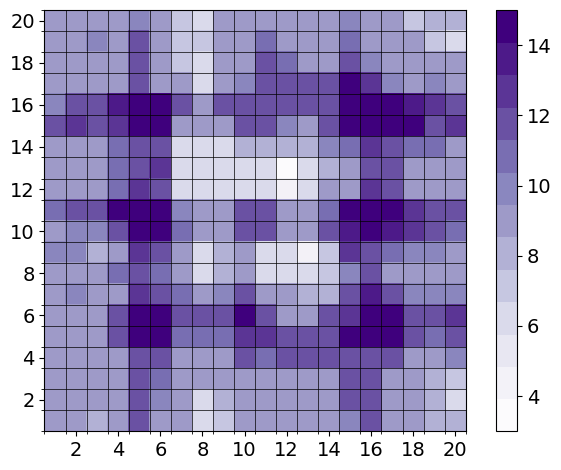}  
	\end{subfigure}
	\caption{Left: Number of affine coefficients in each patch problem $T$ for $n_H=20$ and $\kl=3$.
		The thermal block structure of $A^{1}_\xi$ and $A^{2}_\xi$ is highlighted in green, two patch instances are highlighted in red. Right: local corrector RB sizes of the Stage~1 models in Method 4.}
	\label{fig:aff_coefs}
\end{figure}

We emphasize that for this experiment,
an approximation error of the PG--LOD in $\Jhat$ is still observable, e.g., at the desired parameter $\mu^{\text{d}}$, we have
$
| \Jhat_h^{\text{loc}}(\mu^{\text{d}}) - \Jhat_h(\mu^{\text{d}})|/\Jhat_h(\mu^{\text{d}}) = 8.25 \cdot 10^{-6}.
$
Although this violates \cref{asmpt:truth}, we can expect that all methods converge up to the LOD discretization error, which is sufficiently close for this experiment.

In \cref{fig:aff_coefs}(left), we visualize the number of affine components of the local corrector models directly associated with the number of  components in $A_\mu$, which can be determined by the number of thermal blocks that lie in the patch.
The thermal blocks are highlighted in green, and since~$\kl=3$, the resulting affine components can be counted.
For instance, the lower-left element's patch only reaches the lower-left thermal block (resulting in $2$ affine components). Moreover, the elements in the interior have a patch that reaches up until all $9$ neighboring blocks (resulting in $18$ affine components each).
The discussed patches are highlighted in red in \cref{fig:aff_coefs}.
We conclude that the corrector problems have a more minor parameter dependence than globalized RB methods.
In turn, we can expect the local RB models to require fewer basis functions.

In \cref{fig:aff_coefs}(right), the local RB size of the Stage~1 models in Method 4 is depicted.
It can be seen that the model requires a relatively rich space at the coarse elements that are close to the "jumps" in the desired parameter, cf~\cref{fig:thermal_block}.
As expected, the low conductivity blocks in the middle of the domain do not require many RB enrichments since the optimization problem in these blocks is less demanding.
In addition, from solely looking at \cref{fig:aff_coefs}(left), one would guess that the local patch problems that admit the highest number of affine components require the most basis functions.
The fact that this expectation is invalid proves that the optional enrichment can play a significant role in the algorithm.

All compared methods indeed converged up to the chosen FOC-tolerance to the same point, and it was verified that the point is indeed a local optimum.
We intentionally stopped the FEM-based methods earlier (with $\tau_\text{{FOC}}^\text{{FEM}} = 10^{-4}$) to obtain a comparable optimization error $e^{\Jhat,\text{rel}}(\bar{\mu})$ of up to $10^{-6}$, which is due to the known LOD-error mentioned above.
In \cref{tab:TRRBLOD_1}, \cref{tab:TRRBLOD_2}, and \cref{fig:J_error}, we report relevant information on the evaluation counts, the iteration, and the observed run times.
We note again that the run times include all computational costs until convergence of the algorithm (including all offline expenses).

\begin{table}[h] \centering \footnotesize
	\begin{tabular}{|l|c|c|c|c|c|c|c|c|c|}
		\hline
		& \multicolumn{2}{c|}{} & \multicolumn{2}{c|}{LOD} 	& \multicolumn{2}{c|}{Stage 1} & TS & & \\ \hline
		Evaluations 		& FEM 		   & RB  & Coarse		& Local 	  & Coarse 	& Local & & Outer iter. & Time 
		\\ \hline  \hline
		Cost factor   & \#$\gridh$     & $N_{\text{RB}}$ & \#$\Gridh$ & \#$U(T_{H})_{h}$ &  \#$\Gridh$ & $N_{\text{RB}}$  & $N_{\text{RB}}$ & & \\ \hline
		1.\hphantom{a$\star$} FEM 	    		
		& {163} & -   & -   & - 	& -    	& -   	 & -  & 52 & 6412s 	    \\
		2.\hphantom{c$\star$} TR-RB   		
		& {10}  & 1067 & -   & - & -    	& -   	 & -  & 3  & \hphantom{0}863s \\
		2.r\hphantom{$\star$} R-TR-RB   		
		& {8}  & 1082 & -   & -  & -    	& -   	 & -  & 2  & \hphantom{0}565s	\\
		2.r$\star$ R-TR-RB   		
		& {8}  & 696 & -   & -  & -    	& -   	 & -  & 2  & \hphantom{0}289s	\\
		3.\hphantom{a$\star$} PG-LOD 			
		& -	   & -   & 242 & {128000}& -   	& -   	 & -  & 79 & \hphantom{0}723s \\
		4.\hphantom{c$\star$} TR-TS  		
		& -  & -   & 18   & {19200} & {42} 	& 67200  & {572}	& 6  & \hphantom{0}902s \\
		4.r\hphantom{$\star$} R-TR-TS 		
		& -  & -   & 12   & {12800}	& {20}  	& 32000  & {671}	& 4  & 		\hphantom{0}788s  \\
		4.r$\star$ R-TR-TS 		
		& -  & -   & 12   & {12800}	& {20}  	& 32000  & {311}	& 4  & 		\hphantom{0}281s  \\
		\hline
	\end{tabular}
	\caption{Experiment 1: Evaluations and timings of selected methods. The FEM-based methods are stopped such that $e^{\Jhat,\text{rel}}$ is comparable for all methods. Evaluation counts exclude estimator training.}
	\label{tab:TRRBLOD_1}
\end{table}

 \begin{table}[h] \centering \footnotesize
	\begin{tabular}{|l|c|c|c|c|c|c|c|c|} 
		\hline
		& 		& 		  & \multicolumn{2}{c|}{Online} 	& \multicolumn{3}{c|}{Offline} & \\ \hline
		Method 				  & Total 			  & Speedup & Outer & Inner 					
		& FEM & Stage 1	 &  Stage 2  & $e^{\Jhat_h,\text{rel}}$ 
		\\ \hline  \hline
		1.\hphantom{a$\star$} FEM 	  & 6412s 			  & -  &  6412s			& -  			  
		& - & - & - & 4.24e-06 	\\
		2.\hphantom{c$\star$} TR-RB 			  & \hphantom{0}862s & 7  &  \hphantom{0}841s	& \hphantom{0}22s
		& 841s & - & - & 7.29e-08 	\\
		2.r\hphantom{$\star$} R-TR-RB			  & \hphantom{0}565s & 11 &  \hphantom{0}542s	& \hphantom{0}23s
		& 542s & - & - & 2.45e-07 	\\
		2.r$\star$ R-TR-RB			  & \hphantom{0}290s & 22 &  \hphantom{0}268s	& \hphantom{0}22s
		& 268s & - & - & 2.45e-07 	\\
		3.\hphantom{a$\star$} PG-LOD & \hphantom{0}723s & 9 &  \hphantom{0}723s	& - 			 
		& - & - & - & 4.22e-06 	\\
		4.\hphantom{c$\star$} TR-TS			  & \hphantom{0}902s & 7 &  \hphantom{0}656s	& 246s 			  
		& - &456s & 128s & 4.22e-06 	\\
		4.r\hphantom{$\star$} R-TR-TS			  & \hphantom{0}789s & 8 &  \hphantom{0}499s	& 290s 			  
		& - &378s & \hphantom{0}72s & 4.22e-06 	\\
		4.r$\star$ R-TR-TS			  & \hphantom{0}282 & 23 &  \hphantom{0}276s	& \hphantom{00}5s 
		& - & 175s & \hphantom{0}61s & 4.22e-06 	\\
		\hline
	\end{tabular}
	\caption{Experiment 1: More details on the run times and accuracy of selected methods. The FEM-based methods are stopped such that the $e^{\Jhat,\text{rel}}(\bar{\mu})$ is comparable for all methods.}
	\label{tab:TRRBLOD_2}
\end{table}

\begin{figure}[h!]
	\centering\footnotesize
	\hspace{-0.6cm}
\begin{tikzpicture}

\definecolor{color0}{rgb}{0.65,0,0.15}
\definecolor{color1}{rgb}{0.84,0.19,0.15}
\definecolor{color2}{rgb}{0.96,0.43,0.26}
\definecolor{color3}{rgb}{0.99,0.68,0.38}
\definecolor{color4}{rgb}{1,0.88,0.56}
\definecolor{color5}{rgb}{0.67,0.85,0.91}

\begin{axis}[
  name=right,
legend cell align={left},
width=12cm,
height=6cm,
legend style={fill opacity=0.8, draw opacity=1, text opacity=1, at={(1.05,0.5)}, anchor=west, draw=white!80!black},
log basis y={10},
tick align=outside,
tick pos=left,
x grid style={white!69.0196078431373!black},
xlabel={time in seconds [s]},
xmajorgrids,
xmin=-10., xmax=2000,
xtick style={color=black},
y grid style={white!69.0196078431373!black},
ylabel={\(\displaystyle e^{\Jhat_h,\text{rel}}\)},
ymajorgrids,
ymin=1e-08, ymax=0.1,
ymode=log,
ytick style={color=black}
]
\addplot [semithick, color0, mark=triangle*, mark size=3, mark options={solid, fill opacity=0.5}]
table {%
	200.070465087891 0.0507441622863822
	319.514115810394 0.0206452854274428
	439.002683401108 0.0122548783687695
	558.415676355362 0.00881934203447243
	677.864444732666 0.00843580399321531
	797.298424243927 0.00834282463100888
	916.767745018005 0.00822601478720175
	1036.29923462868 0.00806585705286644
	1155.8025124073 0.00783596459484093
	1275.26995038986 0.00764890826114861
	1394.75429296494 0.00757146507234197
	1514.19849824905 0.00750989253031875
	1633.69614315033 0.00740877961955366
	1753.17409420013 0.00715054203500731
	1872.60409832001 0.00660900373173123
	1992.02441740036 0.00513252875110171
	2111.50488805771 0.0017039027253436
	2230.9416179657 0.000834559352331121
	2350.4620912075 0.00055390822954271
	2469.95707607269 0.000515361491122412
	2589.48239445686 0.000480023309817446
	2708.90225768089 0.000412647088049667
	2828.25652432442 0.000331802692126626
	2947.71574950218 0.000313513195936865
	3067.23507118225 0.000306179854902666
	3186.76052379608 0.000302984719787069
	3306.35480165482 0.000299088196630493
	3426.02895641327 0.000292383612388925
	3545.5935344696 0.000283129261227977
	3665.14190077782 0.000269771735050162
	3784.65148234367 0.000241344531423993
	3904.1110496521 0.00012132993506242
	4023.62576079369 8.55881203949416e-05
	4143.0657954216 6.33011351940382e-05
	4262.40058970451 4.74035410527041e-05
	4381.74972939491 3.29545392687702e-05
	4501.13214826584 2.21987431867632e-05
	4620.54872393608 1.88058961714255e-05
	4739.99843406677 1.57931258744082e-05
	4859.50440311432 1.14829687869467e-05
	4978.91410398483 9.44278331882487e-06
	5098.30205607414 7.29588082082344e-06
	5217.59353542328 7.1334375806309e-06
	5337.03256177902 6.85962346436142e-06
	5456.45532035828 6.57101363743529e-06
	5575.91868066788 6.12651232168027e-06
	5695.32953858376 5.79558165947347e-06
	5814.73174095154 5.58427174746079e-06
	5934.09788155556 5.44985821426813e-06
	6053.51430130005 5.29632589296725e-06
	6172.91858363152 5.15706750570111e-06
	6292.31639909744 4.820948501294e-06
	6411.70653796196 4.23805708305025e-06
};
\addlegendentry{1.\hphantom{r$\star$} FOM FEM}
\addplot [semithick, color2, mark=diamond*, mark size=3, mark options={solid, fill opacity=0.5}]
table {%
	169.244106481783 0.0507441622863822
	344.878037641756 0.00376830011611018
	593.232571790926 5.17272768547627e-05
	862.681221912615 7.28546343253811e-08
};
\addlegendentry{2.\hphantom{r$\star$} TR-RB}
\addplot [semithick, color1, mark=asterisk, mark size=3, mark options={solid,rotate=180, fill opacity=0.5}]
table {%
	83.9604925354943 0.0507441622863822
	182.33541767206 0.00192349081821419
	289.984329386614 2.45332196557868e-07
};
\addlegendentry{2.r$\star$ R-$\star$-TR-RB}
\addplot [semithick, color5, mark=triangle*, mark size=3, mark options={solid, fill opacity=0.5}]
table {%
	8.36511540412903 0.0507441622863822
	16.9066307544708 0.0209219309572612
	26.1908042430878 0.0124035166195799
	35.0032784938812 0.0088908691043017
	43.6283776760101 0.0084948168428487
	52.6025915145874 0.00839890043853209
	61.1713907718658 0.00828744263310321
	69.9955928325653 0.008142682584523
	78.7149248123169 0.00792265267685432
	87.5153856277466 0.00773221493102105
	96.8462755680084 0.00764925600733113
	105.764570951462 0.00758739803358344
	115.326550006866 0.00749145629760894
	124.497675657272 0.00724390700242972
	133.284508943558 0.00673506520897571
	142.172802448273 0.00546705271709724
	150.972652435303 0.00177307422590722
	160.169446229935 0.000878273478293501
	168.507457017899 0.000594669657495217
	177.501410245895 0.000545449098740658
	186.040268659592 0.000499567165252124
	194.682205915451 0.000418439561376882
	203.370742559433 0.000334000869457451
	212.157036304474 0.000321128354331623
	220.845720767975 0.000312748661498663
	229.76584482193 0.000309253913873375
	238.486796617508 0.000296773702099129
	247.670082092285 0.000280736535533732
	256.582405805588 0.000252006235402336
	264.886821746826 0.000221640781500554
	273.337299108505 0.000191252830927358
	282.136838674545 0.000146487560722663
	291.709166765213 0.000121209982339998
	300.475846529007 0.000106766253267887
	310.368748903275 8.90778661957192e-05
	319.668181657791 7.14050084105189e-05
	328.630468130112 5.33272487641678e-05
	337.131402492523 4.14556047585446e-05
	347.184409379959 3.60349223136325e-05
	356.040106296539 3.28281483947812e-05
	364.796132564545 2.64248834642711e-05
	373.928318738937 1.9550416066938e-05
	383.249375104904 1.37154721091548e-05
	391.947289943695 1.11278334731235e-05
	401.387892723083 1.03071188006787e-05
	411.080298900604 1.00055220781936e-05
	419.910415649414 9.77069973484568e-06
	428.981688737869 9.43300717937134e-06
	438.35834479332 8.74485100754718e-06
	447.360706329346 7.64349348525428e-06
	457.317383527756 6.27347232007303e-06
	466.316675901413 5.06978976755335e-06
	475.049007177353 4.87413750027699e-06
	483.79289150238 4.7900847979232e-06
	492.212411165237 4.71377772504056e-06
	501.443563938141 4.67002575965658e-06
	510.568825483322 4.66023049661146e-06
	519.054765224457 4.65553105999383e-06
	528.241171360016 4.6431838203187e-06
	536.89192700386 4.61694896092268e-06
	546.548330783844 4.56804921755705e-06
	554.862567424774 4.51034417414675e-06
	563.634038686752 4.47279533077882e-06
	572.71826004982 4.45679167215829e-06
	581.681055307388 4.44397494181459e-06
	591.303342580795 4.41847566090559e-06
	600.803759813309 4.37306434664109e-06
	609.934619903564 4.31747473950139e-06
	619.172713279724 4.28525687046211e-06
	628.538497924805 4.27833537597344e-06
	637.759961366653 4.27771610489458e-06
	647.418738126755 4.27748495690494e-06
	656.542580842972 4.27662146518237e-06
	666.285809755325 4.27473444686299e-06
	675.468652963638 4.26988039858855e-06
	684.309336662292 4.25965509109894e-06
	694.790027618408 4.24265167398019e-06
	704.296395778656 4.22626406981763e-06
	713.701845407486 4.21943978334838e-06
	722.969966173172 4.21846934051118e-06
};
\addlegendentry{3.\hphantom{r$\star$} FOM LOD}
\addplot [semithick, color3, mark=diamond*, mark size=3, mark options={solid, fill opacity=0.5}]
table {%
	96.6443058233708 0.0507441622863822
	216.95975577645 0.0325917495587029
	347.332955954596 0.00588037324303126
	567.091648219153 1.17817949600862e-05
	754.58963858895 4.33649335973563e-06
	864.28984558396 4.21903395486645e-06
	901.763541577384 4.21825361662798e-06
};
\addlegendentry{4.\hphantom{r$\star$} TR-TSRBLOD}
\addplot [semithick, black, mark=asterisk, mark size=3, mark options={solid, fill opacity=0.5}]
table {%
	56.700649747625 0.0507441622863822
	116.08456541039 0.00861268259446701
	195.563434610143 5.618549788311e-06
	271.66650606133 4.23929878623497e-06
	281.519287595525 4.21827631891247e-06
};
\addlegendentry{4.r$\star$ R-$\star$-TR-TSRBLOD}
\addplot [semithick, black, dashed, mark size=0, mark options={solid, fill opacity=0.5}]
table {%
	-20 4.21825361662798e-06
	2100 4.21825361662798e-06
};
\addlegendentry{LOD discretization error}
\end{axis}
\end{tikzpicture} 	
	\caption{Experiment 1: Error decay and performance of selected algorithms. Method 1 converged after 6412s.
	}
	\label{fig:J_error}
\end{figure}
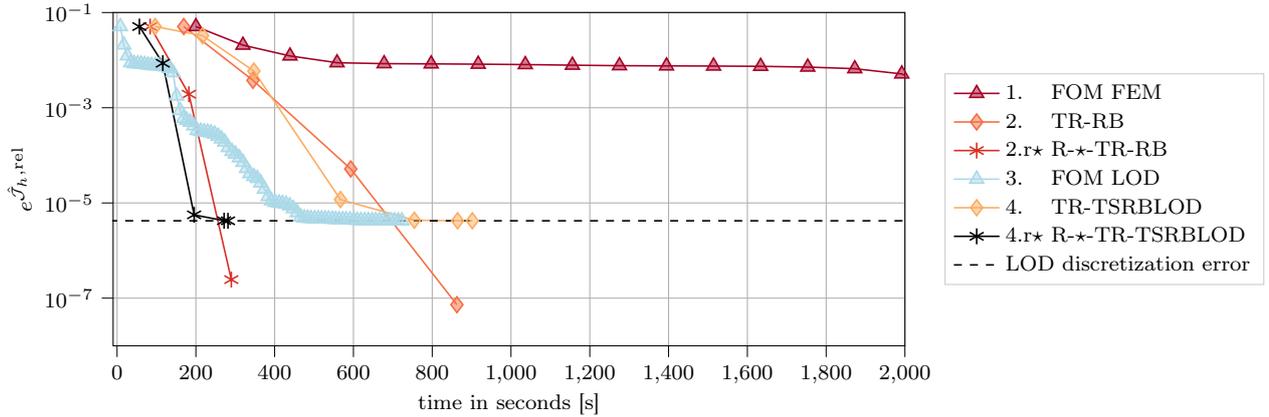

We conclude that all Methods~2-4 give a significant speedup to the standard FEM Method~1.
Although~$52$ iterations of Method~1 and the corresponding $163$ FEM evaluations are relatively few for a $32$-dimensional optimization problem, the method suffers from the computational cost for performing FEM solutions with $1.4$ Mio. DoFs.
As already shown in \cite{paper1}, the TR-RB Method~2 is mainly designed to avoid these expensive FEM evaluations.
The (relaxed) TR-RB methods converge already after few outer iterations, which only requires~$10$ and $8$ FEM-based enrichments of the reduced spaces.
On the other hand, the inner RB evaluations are cheap.

As expected, the localized methods only converge until the priorly known approximation error of the PG--LOD is reached.
However, it can be seen that the TR-TSRBLOD methods find the same point and are not subject to approximation issues.

A significant reason why the TSRBLOD method is particularly suitable for this work is its very efficient online phase.
This result can be observed in \cref{tab:TRRBLOD_2}, where extended timings are given for the TSRBLOD method.
Just as the TR-RB methods, only a few seconds are required to solve the sub-problems in the 4.r$\star$ variant which is independent of the coarse LOD mesh.
However, the sub-problem is more demanding for the variant where the estimator is used, which is due to the fact that we did not afford the offline time to prepare for the two-scale error estimator in Stage~2.
We further notice that the localized methods show a comparably good convergence speed w.r.t. the FEM-based methods, although FEM is still comparably fast.
We also see that Method~3 (localized FOM) shows a strong convergence speed.
This is due to the relatively small patch problems such that the localized corrector problems and the corresponding Stage~2 reduction do not pay off immensely.
We also emphasize that Method~3 and 4 immensely benefit from the parallelization.

The relaxed versions of the TR methods show the fastest convergence behavior in this experiment.
The fully enforced certification in the non-relaxed TR-RB and TR-TSRBLOD can not detect the full benefit from their respective surrogate model and, instead, truncate the sub-problems too early.
Comparing the R-TR and the R-$\star$-RB variants, we observe that our choice of the relaxation parameters, the relaxed TR methods unconditionally trust the used surrogate models.
The $\star$-variants, where estimation is completely left out for early iterations, show that including the estimation does not change the result but only increases the computational time due to the pre-assembly preparation and evaluation of the estimates.

In conclusion, FEM based-methods can reliably be replaced by localized methods already for moderately small fine-mesh sizes.
The accuracy of the localized method can be expected up to the LOD-discretization error (cf.~the discussion above regarding \cref{asmpt:truth}).
The full benefit of the TR-TSRBLOD approaches in comparison to the localized FOM can only be deduced for scenarios where the PG--LOD is costly in itself.
Thus, in the second experiment, we increase the complexity of the multiscale structure of the problem.

\subsection{Experiment 2: Large scale example}

We consider a large scale example where the global FEM mesh does not fit into the machine's memory.
We set the multiscale resolution to $N_1 = 1.000$ and $N_2 = 250$.
For the fine mesh, we again aim for at least~$4$ fine mesh entities in each multiscale cell and hence choose $n_h = 4000$.
Therefore, the FEM mesh would have~$16$ Mio degrees of freedom, which we consider prohibitively large.
Thus, we do not utilize FEM-based methods and only compare Methods 3 and 4, where we only use Method 4.r$\star$ as relaxed variant since it has proven advantageous in the former experiment.
For the coarse-grid, we choose $n_H=40$, which results in $1600$ coarse grid cells and, in particular, $\kl=4$ and $810.000$ fine-mesh elements for full patches $U_\kl(T)$.
Since FEM evaluations are not available, the desired solution $u_{\mu^{\text{d}}}$ is computed with the PG--LOD, i.e. we solve~\eqref{eq:PG_LOD} for~$\mu^{\text{d}}$.

Similar to the above illustrations, in \cref{fig:aff_coefs_2}, we report the respective number of affine components of the patch problems as well as the final local RB sizes of the certified TR-TSRBLOD method with optional enrichment (Method 4).
In particular, \cref{fig:aff_coefs_2} can be interpreted as the refined version of \cref{fig:aff_coefs}, where it is even more visible that the local corrector problems have fewer affine components and require more basis functions for the corrector problems that are largely affected by the "jumps" in the desired thermal block state, depicted in \cref{fig:thermal_block}.
Just as before, it can be seen that the amount of basis functions is also associated with the intensity of the respective "jumps", and the low conductivity in the middle of the domain is well visible.

In \cref{tab:TRRBLOD_3} and \cref{tab:TRRBLOD_4}, we again provide an extensive comparison concerning evaluations, run time, and iteration counts of the methods.
It can be seen that the TSRBLOD-based methods successfully reduce the computational effort of Method 3, which is mainly due to the increasing number of fine-mesh DoFs in the patches.
With increasing complexity of the multiscale problem, we thus expect even more speedups.
We also note that the speedup w.r.t. the FOM method is also dependent on the outer iteration counts, cf. \cite{paper1,paper2}.
It can be expected that the benefit of reduced models is even more present for increasing complexity of the optimization problem.

\begin{figure}
	\centering
	\begin{subfigure}{.48\textwidth}
		\centering
		\includegraphics[width=.8\linewidth]{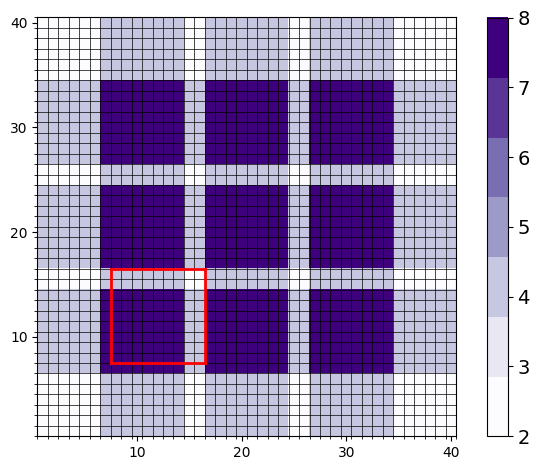}  
	\end{subfigure}
	\begin{subfigure}{.48\textwidth}
		\centering
		\includegraphics[width=.8\linewidth]{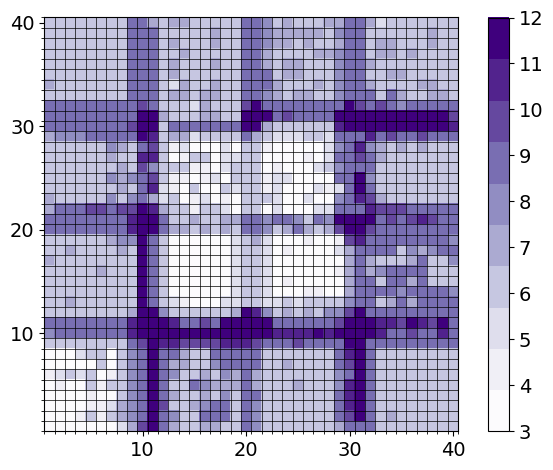}  
	\end{subfigure}
	\caption{Left: Number of affine coefficients in the patch problems for $n_H=40$ and $\kl=4$.
		One patch is highlighted in red. Right: local RB sizes of the Stage~1 models in Method~4.}
	\label{fig:aff_coefs_2}
\end{figure}

 \begin{table}[h!] \centering \footnotesize
 	\begin{tabular}{|l|c|c|c|c|c|c|c|c|c|} 
 		\hline
 		& \multicolumn{1}{c|}{} & \multicolumn{2}{c|}{LOD} 	& \multicolumn{2}{c|}{Stage 1} & TS &
 		\multicolumn{3}{c|}{} \\ \hline
 		Evaluations 		& FEM 		   & Coarse		& Local 	  & Coarse 	& Local & & Outer iter. & Time 
 		& $e^{\Jhat_h^\textnormal{loc},\text{rel}}$
 		\\ \hline  \hline
 		Cost factor   & \#$\gridh$     & \#$\Gridh$ & \#$U(T_{H})_{h}$ &  \#$\Gridh$ & $N_{\text{RB}}$  & $N_{\text{RB}}$ &
 		\multicolumn{3}{c|}{} \\ \hline
 		3.\hphantom{a$\star$} PG-LOD
 		& -  		     & 307 & {665600}& -   	& -   	 & -  & 100 & 11317s
 		& 2.07e-10
 		\\
 		4.\hphantom{c$\star$} TR-TS
 		& -  		     & 10   & {32000}	& {20}  	& 128000  & {938} & 5  & \hphantom{0}5393s
 		& 1.57e-10
 		\\
 		4.r$\star$ R-TR-TS
 		& -  		     & 8   & {51200}	& {30}  	& 128000  & {422} & 4  & \hphantom{0}2998s
 		& 1.32e-11
 		\\
 		\hline
 	\end{tabular}
  	\caption{Experiment 2: Evaluations and accuracy of selected methods. Evaluation counts exclude estimator training.}
 	\label{tab:TRRBLOD_3}
 \end{table}

 \begin{table}[h!] \centering \footnotesize
	\begin{tabular}{|l|c|c|c|c|c|c|c|} 
		\hline
		& 		& 		  & \multicolumn{2}{c|}{Online} 	& \multicolumn{3}{c|}{Offline}  \\ \hline
		Method 				  & Total 			  & Speedup & Outer & Inner 					
		& FEM & Stage 1	 &  Stage 2  
		\\ \hline  \hline
		3.\hphantom{a$\star$} PG-LOD & 11317s & - &  11317s	& - 			 
		& - & - & - 
		\\
		4.\hphantom{c$\star$} TR-TS & \hphantom{0}5393s & 2 &  \hphantom{000}486s	& 360s  & - & 4042s & 505s 
		\\
		4.r$\star$ R-TR-TS			 & \hphantom{0}2998s & 4 &  \hphantom{000}5s	& \hphantom{00}5s & - & 2690s & 336s
		\\
		\hline
	\end{tabular}
	\caption{Experiment 2: More details on the run times of selected methods.}
	\label{tab:TRRBLOD_4}
\end{table}

\section{Concluding remarks and future work}

In this article, we presented a first combination of localized reduced basis methods for efficiently solving parameterized multiscale problems with optimization methods that adaptively construct such localized reduced models in the context of an iterative error-aware trust-region algorithm for accelerating PDE-constrained optimization.
Moreover, we have formulated a relaxed version of the TR algorithm from \cite{paper1} to neglect the strong certification in the first iterations.
For this relaxation, the same convergence result holds.

Subsequently, we have discretized the optimality system of the PDE-constrained optimization problem~\eqref{P} with a localized ansatz based on the Petrov-Galerkin version of the localized orthogonal decomposition method.
For an online efficient reduced model with optional local basis enrichment, such that the sub-problems of the TR algorithm can be solved fast, we have used the TSRBLOD based on a two-scale RB ansatz of the LOD.

The TR-TSRRBLOD method has proven advantageous both in terms of computational effort and adaptivity concerning the localized RB models.
In the experiments, we showed that localized RB approaches can efficiently replace FEM-based techniques, especially for growing complexity of the multiscale system.

Many tasks have been left for the future.
Although the underlying multiscale data is already highly heterogeneous and non-periodic, and the LOD approach showed good approximation properties w.r.t. FEM, it is commonly known that the LOD struggles, e.g., for high-contrast problems or complex coarse data such as thin channels.
For using the TR-TSRBLOD, it has to be verified priorly that \cref{asmpt:truth} is given up to an acceptable tolerance.
To remedy this, the discussed concepts can be generalized to other multiscale methods, always dependent on the respective multiscale task.
It is also desirable to derive a posteriori error theory for the LOD such that the homogenization term from \eqref{eq:hom_red_term} can be used to validate the approximation quality of the LOD, cf.~\cref{rem:trunc-red-hom-term}.
We also mention that, in this work, we have enforced several problem assumptions, e.g., ellipticity, symmetry, and homogeneous boundary conditions, to simplify the presentation.
However, it seems straightforward to generalize the methodology to more challenging problem classes.

Concerning the specific instance of the TR-TSRBLOD, the numerical experiments already showed an overall speedup w.r.t. FEM and the PG--LOD, though the first experiment was relatively small.
Our theoretical findings suggest even better run times within a more HPC-oriented implementation.
Moreover, an intermediate preparatory reduction of the two-scale system can be used to decrease further offline expenses of Stage~2.
The described TR-TSRBLOD method can also be enhanced in terms of the choice of the local enrichment tolerance $\tau_{{\text{loc}}}$, such that an appropriate choice for the respective optimization problem or model can efficiently be found with adaptive refinements, cf.~\cref{sec:local_basis_enrichment}.

In our numerical experiments, we observed that the relaxed variant does not use the estimator (cf.~R-$\star$-RB variant vs. R-RB variant in \cref{sec:exp_1}).
However, ignoring the estimator for all relaxed iterations can not be considered valid in general.
In a more general context, the outer iterations where the estimator can be ignored can be found prior to the algorithm by evaluating the estimator for an empty basis at random parameter samples and computing the maximum value.

\section*{Code availability}
All experiments have been implemented in \texttt{Python} using \texttt{gridlod} \cite{gridlod} for the PG--LOD discretization and \texttt{PyMOR} \cite{pymor} for the model order reduction.
In particular, the software is internally based on the software that has been used in \cite{paper1} and \cite{tsrblod}.
The complete source code for all experiments, including setup instructions, can be found in~\cite{CodeTRLRB}, also available under \url{https://github.com/TiKeil/Trust-region-TSRBLOD-code}.

{\small

\begin{thebibliography}{10}
	
	\bibitem{AB12}
	A.~Abdulle and Y.~Bai.
	\newblock Reduced basis finite element heterogeneous multiscale method for
	high-order discretizations of elliptic homogenization problems.
	\newblock {\em Journal of Computational Physics}, 231(21):7014--7036, 2012.
	
	\bibitem{AB13}
	A.~Abdulle and Y.~Bai.
	\newblock Adaptive reduced basis finite element heterogeneous multiscale
	method.
	\newblock {\em Computer Methods in Applied Mechanics and Engineering},
	257:203--220, 2013.
	
	\bibitem{AB14_2}
	A.~Abdulle and Y.~Bai.
	\newblock Reduced-order modelling numerical homogenization.
	\newblock {\em Philosophical Transactions of the Royal Society A: Mathematical,
		Physical and Engineering Sciences}, 372(2021):20130388, 2014.
	
	\bibitem{AB14}
	A.~Abdulle, Y.~Bai, and G.~Vilmart.
	\newblock An offline--online homogenization strategy to solve quasilinear
	two-scale problems at the cost of one-scale problems.
	\newblock {\em International Journal for Numerical Methods in Engineering},
	99(7):469--486, 2014.
	
	\bibitem{A19}
	A.~Abdulle and A.~D. Blasio.
	\newblock Numerical homogenization and model order reduction for multiscale
	inverse problems.
	\newblock {\em Multiscale Modeling \& Simulation}, 17(1):399--433, 2019.
	
	\bibitem{RBLOD}
	A.~Abdulle and P.~Henning.
	\newblock A reduced basis localized orthogonal decomposition.
	\newblock {\em Journal of Computational Physics}, 295:379--401, 2015.
	
	\bibitem{Allaire2012}
	G.~Allaire.
	\newblock {\em Shape optimization by the homogenization method}, volume 146.
	\newblock Springer Science \& Business Media.
	
	\bibitem{ADFM2016}
	G.~Allaire, C.~Dapogny, A.~Faure, and G.~Michailidis.
	\newblock Shape optimization of a layer by layer mechanical constraint for
	additive manufacturing.
	\newblock working paper or preprint, Nov. 2016.
	
	\bibitem{AHP21}
	R.~Altmann, P.~Henning, and D.~Peterseim.
	\newblock Numerical homogenization beyond scale separation.
	\newblock {\em Acta Numer.}, 30:1--86, 2021.
	
	\bibitem{BL2011}
	I.~Babuska and R.~Lipton.
	\newblock Optimal local approximation spaces for generalized finite element
	methods with application to multiscale problems.
	\newblock {\em Multiscale Model. Simul.}, 9(1):373--406, 2011.
	
	\bibitem{BL+2020}
	I.~Babu\v{s}ka, R.~Lipton, P.~Sinz, and M.~Stuebner.
	\newblock Multiscale-spectral {GFEM} and optimal oversampling.
	\newblock {\em Comput. Methods Appl. Mech. Engrg.}, 364:112960, 28, 2020.
	
	\bibitem{paper2}
	S.~Banholzer, T.~Keil, M.~Ohlberger, L.~Mechelli, F.~Schindler, and
	S.~Volkwein.
	\newblock An adaptive projected {N}ewton non-conforming dual approach for
	trust-region reduced basis approximation of {PDE}-constrained parameter
	optimization.
	\newblock {\em Pure Appl. Funct. Anal.}, 7(5):1561--1596, 2022.
	
	\bibitem{BarraultMadayEtAl2004}
	M.~Barrault, Y.~Maday, N.~C. Nguyen, and A.~T. Patera.
	\newblock An `empirical interpolation' method: application to efficient
	reduced-basis discretization of partial differential equations.
	\newblock {\em C. R. Math.}, 339(9):667--672, 2004.
	
	\bibitem{MR3672144}
	P.~Benner, A.~Cohen, M.~Ohlberger, and K.~Willcox, editors.
	\newblock {\em Model reduction and approximation}, volume~15 of {\em
		Computational Science \& Engineering}.
	\newblock SIAM, Philadelphia, PA, 2017.
	\newblock Theory and algorithms.
	
	\bibitem{Boy2008}
	S.~Boyaval.
	\newblock Reduced-basis approach for homogenization beyond the periodic
	setting.
	\newblock {\em Multiscale Model. Simul.}, 7(1):466--494, 2008.
	
	\bibitem{brown2016multiscale}
	D.~L. Brown and D.~Peterseim.
	\newblock A multiscale method for porous microstructures.
	\newblock {\em Multiscale Modeling \& Simulation}, 14(3):1123--1152, 2016.
	
	\bibitem{Buh2017}
	A.~{Buhr}.
	\newblock {Exponential Convergence of Online Enrichment in Localized Reduced
		Basis Methods}.
	\newblock {\em IFAC-PapersOnLine}, 51(2):302--306, 2018.
	
	\bibitem{BIORSS21}
	A.~Buhr, L.~Iapichino, M.~Ohlberger, S.~Rave, F.~Schindler, and K.~Smetana.
	\newblock Localized model reduction for parameterized problems, 2021.
	\newblock In Benner, et.al.. Model Order Reduction. Volume 2. Walter De Gruyter
	GmbH, Berlin, 2021.
	
	\bibitem{CEGG2014}
	V.~M. Calo, Y.~Efendiev, J.~Galvis, and M.~Ghommem.
	\newblock Multiscale empirical interpolation for solving nonlinear {PDE}s.
	\newblock {\em J. Comput. Phys.}, 278:204--220, 2014.
	
	\bibitem{CS2010}
	S.~Chaturantabut and D.~C. Sorensen.
	\newblock Nonlinear model reduction via discrete empirical interpolation.
	\newblock {\em SIAM J. Sci. Comput.}, 32(5):2737--2764, 2010.
	
	\bibitem{MR3529770}
	R.~E. Christiansen and O.~Sigmund.
	\newblock Designing meta material slabs exhibiting negative refraction using
	topology optimization.
	\newblock {\em Struct. Multidiscip. Optim.}, 54(3):469--482, 2016.
	
	\bibitem{CEL2014}
	E.~T. Chung, Y.~Efendiev, and G.~Li.
	\newblock An adaptive gmsfem for high-contrast flow problems.
	\newblock {\em J. Comput. Phys.}, 273:54--76, 2014.
	
	\bibitem{MR3875293}
	S.~Conti, B.~Geihe, M.~Lenz, and M.~Rumpf.
	\newblock {A posteriori} modeling error estimates in the optimization of
	two-scale elastic composite materials.
	\newblock {\em ESAIM Math. Model. Numer. Anal.}, 52(4):1457--1476, 2018.
	
	\bibitem{DHO2012}
	M.~Drohmann, B.~Haasdonk, and M.~Ohlberger.
	\newblock Reduced basis approximation for nonlinear parametrized evolution
	equations based on empirical operator interpolation.
	\newblock {\em SIAM J. Sci. Comput.}, 34:A937--A969, 2012.
	
	\bibitem{EE2005}
	W.~E and B.~Engquist.
	\newblock The heterogeneous multi-scale method for homogenization problems.
	\newblock In {\em Multiscale methods in science and engineering}, volume~44 of
	{\em Lect. Notes Comput. Sci. Eng.}, pages 89--110. Springer, Berlin, 2005.
	
	\bibitem{Efendiev2013}
	Y.~Efendiev, J.~Galvis, and T.~Y. Hou.
	\newblock Generalized multiscale finite element methods (gmsfem).
	\newblock {\em Journal of computational physics}, 251:116--135, 2013.
	
	\bibitem{Efendiev2009}
	Y.~Efendiev and T.~Y. Hou.
	\newblock {\em Multiscale finite element methods: theory and applications},
	volume~4.
	\newblock Springer Science \& Business Media.
	
	\bibitem{elf}
	D.~Elfverson, V.~Ginting, and P.~Henning.
	\newblock On multiscale methods in {Petrov-Galerkin} formulation.
	\newblock {\em Numerische Mathematik}, 131(4):643--682, 2015.
	
	\bibitem{MR2143505}
	S.~B. Hazra and V.~Schulz.
	\newblock On efficient computation of the optimization problem arising in the
	inverse modeling of non-stationary multiphase multicomponent flow through
	porous media.
	\newblock {\em Comput. Optim. Appl.}, 31(1):69--85, 2005.
	
	\bibitem{gridlod}
	F.~Hellman and T.~Keil.
	\newblock gridlod.
	\newblock \url{https://github.com/fredrikhellman/gridlod}.
	
	\bibitem{HKM20}
	F.~Hellman, T.~Keil, and A.~M{\aa}lqvist.
	\newblock Numerical upscaling of perturbed diffusion problems.
	\newblock {\em SIAM Journal on Scientific Computing}, 42(4):A2014--A2036, 2020.
	
	\bibitem{MH16}
	F.~Hellman and A.~M{\aa}lqvist.
	\newblock Contrast independent localization of multiscale problems.
	\newblock {\em Multiscale Modeling \& Simulation}, 15(4):1325--1355, 2017.
	
	\bibitem{HM19}
	F.~Hellman and A.~M{\aa}lqvist.
	\newblock Numerical homogenization of elliptic {PDE}s with similar
	coefficients.
	\newblock {\em Multiscale Modeling \& Simulation}, 17(2):650--674, 2019.
	
	\bibitem{HMP2014}
	P.~Henning, A.~M{\aa}lqvist, and D.~Peterseim.
	\newblock A localized orthogonal decomposition method for semi-linear elliptic
	problems.
	\newblock {\em ESAIM Math. Model. Numer. Anal.}, 48(5):1331--1349, 2014.
	
	\bibitem{HOS2014}
	P.~Henning, M.~Ohlberger, and B.~Schweizer.
	\newblock An adaptive multiscale finite element method.
	\newblock {\em Multiscale Model. Simul.}, 12(3):1078--1107, 2014.
	
	\bibitem{Hesthaven2016}
	J.~S. Hesthaven, G.~Rozza, and B.~Stamm.
	\newblock {\em Certified reduced basis methods for parametrized partial
		differential equations}.
	\newblock SpringerBriefs in Mathematics. Springer, Cham; BCAM, Bilbao, {Cham},
	2016.
	\newblock BCAM SpringerBriefs.
	
	\bibitem{HesthavenZhangEtAl2015}
	J.~S. Hesthaven, S.~Zhang, and X.~Zhu.
	\newblock Reduced {{Basis Multiscale Finite Element Methods}} for {{Elliptic
			Problems}}.
	\newblock {\em Multiscale Modeling \& Simulation}, 13(1):316--337, 2015.
	
	\bibitem{HPUU2009}
	M.~Hinze, R.~Pinnau, M.~Ulbrich, and S.~Ulbrich.
	\newblock {\em Optimization with {PDE} constraints}.
	\newblock Springer Netherlands, 2009.
	
	\bibitem{MsFEM}
	T.~Y. Hou and X.-H. Wu.
	\newblock A multiscale finite element method for elliptic problems in composite
	materials and porous media.
	\newblock {\em Journal of computational physics}, 134(1):169--189, 1997.
	
	\bibitem{Hughes1995387}
	T.~J. Hughes.
	\newblock Multiscale phenomena: Green's functions, the dirichlet-to-neumann
	formulation, subgrid scale models, bubbles and the origins of stabilized
	methods.
	\newblock {\em Computer Methods in Applied Mechanics and Engineering},
	127(1--4):387 -- 401, 1995.
	
	\bibitem{HUGHES19983}
	T.~J. Hughes, G.~R. Feijóo, L.~Mazzei, and J.-B. Quincy.
	\newblock The variational multiscale method—a paradigm for computational
	mechanics.
	\newblock {\em Computer Methods in Applied Mechanics and Engineering}, 166(1):3
	-- 24, 1998.
	
	\bibitem{Jansen201140}
	J.~J.D.
	\newblock Adjoint-based optimization of multi-phase flow through porous media -
	a review.
	\newblock {\em Computers and Fluids}, 46(1):40 – 51, 2011.
	\newblock Cited by: 164.
	
	\bibitem{CodeTRLRB}
	T.~Keil.
	\newblock {Software for:A Relaxed Localized Trust-Region Reduced Basis Approach
		for Optimization of Multiscale Problems
		\url{https://doi.org/10.5281/zenodo.7821980}}, 2023.
	
	\bibitem{paper1}
	T.~Keil, L.~Mechelli, M.~Ohlberger, F.~Schindler, and S.~Volkwein.
	\newblock A non-conforming dual approach for adaptive trust-region reduced
	basis approximation of {PDE}-constrained parameter optimization.
	\newblock {\em ESAIM. Mathematical Modelling and Numerical Analysis},
	55(3):1239, 2021.
	
	\bibitem{paper3}
	T.~Keil and M.~Ohlberger.
	\newblock Model reduction for large scale systems.
	\newblock In {\em Large-scale scientific computing}, volume 13127 of {\em
		Lecture Notes in Comput. Sci.}, pages 16--28. Springer, Cham, 2022.
	
	\bibitem{tsrblod}
	T.~Keil and S.~Rave.
	\newblock An online efficient two-scale reduced basis approach for the
	localized orthogonal decomposition.
	\newblock {\em arXiv preprint arXiv:2111.08643}, 2021, accepted for publication
	in SIAM J. Sci. Comput. 2023.
	
	\bibitem{kelley}
	C.~T. Kelley.
	\newblock {\em Iterative methods for optimization}, volume~18.
	\newblock Siam, 1999.
	
	\bibitem{LM2005}
	M.~G. Larson and A.~M{\aa}lqvist.
	\newblock Adaptive variational multiscale methods based on a posteriori error
	estimation: duality techniques for elliptic problems.
	\newblock In {\em Multiscale methods in science and engineering}, volume~44 of
	{\em Lect. Notes Comput. Sci. Eng.}, pages 181--193. Springer, Berlin, 2005.
	
	\bibitem{MS22}
	C.~Ma and R.~Scheichl.
	\newblock Error estimates for discrete generalized {FEM}s with locally optimal
	spectral approximations.
	\newblock {\em Math. Comp.}, 91(338):2539--2569, 2022.
	
	\bibitem{MSD22}
	C.~Ma, R.~Scheichl, and T.~Dodwell.
	\newblock Novel design and analysis of generalized finite element methods based
	on locally optimal spectral approximations.
	\newblock {\em SIAM J. Numer. Anal.}, 60(1):244--273, 2022.
	
	\bibitem{MP14}
	A.~M{\aa}lqvist and D.~Peterseim.
	\newblock Localization of elliptic multiscale problems.
	\newblock {\em Mathematics of Computation}, 83(290):2583--2603, 2014.
	
	\bibitem{LODbook}
	A.~M{\aa}lqvist and D.~Peterseim.
	\newblock {\em Numerical Homogenization by Localized Orthogonal Decomposition}.
	\newblock SIAM, 2020.
	
	\bibitem{pymor}
	R.~Milk, S.~Rave, and F.~Schindler.
	\newblock py{MOR}, {M}odel {O}rder {R}eduction with {P}ython, December 2014.
	
	\bibitem{Maalqvist2014}
	A.~Målqvist and D.~Peterseim.
	\newblock Localization of elliptic multiscale problems.
	\newblock {\em Math. Comp.}, 83(290):2583–2603, Jun 2014.
	
	\bibitem{Nguyen2008}
	N.~C. Nguyen.
	\newblock A multiscale reduced-basis method for parametrized elliptic partial
	differential equations with multiple scales.
	\newblock {\em Journal of Computational Physics}, 227(23):9807--9822, 2008.
	
	\bibitem{Ohl2005}
	M.~Ohlberger.
	\newblock A posteriori error estimates for the heterogeneous multiscale finite
	element method for elliptic homogenization problems.
	\newblock {\em Multiscale Model. Simul.}, 4(1):88--114, 2005.
	
	\bibitem{OS12}
	M.~Ohlberger and M.~Schaefer.
	\newblock A reduced basis method for parameter optimization of multiscale
	problems.
	\newblock In {\em Proceedings of ALGORITMY}, volume 2012, pages 1--10, 2012.
	
	\bibitem{OhlbergerSchaeferEtAl2018}
	M.~Ohlberger, M.~Schaefer, and F.~Schindler.
	\newblock Localized {{Model Reduction}} in {{PDE Constrained Optimization}}.
	\newblock In V.~Schulz and D.~Seck, editors, {\em Shape {{Optimization}},
		{{Homogenization}} and {{Optimal Control}}}, pages 143--163. {Springer},
	{Cham}, 2018.
	
	\bibitem{OS2014}
	M.~Ohlberger and F.~Schindler.
	\newblock A-posteriori error estimates for the localized reduced basis
	multi-scale method.
	\newblock In J.~Fuhrmann and et~al., editors, {\em FVCA VII-Methods and
		Theoretical Aspects}, volume~77 of {\em PROMS}, pages 421--429. Springer,
	2014.
	
	\bibitem{SO15}
	M.~Ohlberger and F.~Schindler.
	\newblock Error control for the localized reduced basis multiscale method with
	adaptive on-line enrichment.
	\newblock {\em SIAM Journal on Scientific Computing}, 37(6):A2865--A2895, 2015.
	
	\bibitem{Ohlberger202029}
	M.~Ohlberger, B.~Schweizer, M.~Urban, and B.~Verf\"urth.
	\newblock Mathematical analysis of transmission properties of electromagnetic
	meta-materials.
	\newblock {\em Networks and Heterogeneous Media}, 15(1):29--56, 2020.
	
	\bibitem{Barbara18}
	M.~Ohlberger and B.~Verf\"urth.
	\newblock A new heterogeneous multiscale method for the helmholtz equation with
	high contrast.
	\newblock {\em Multiscale Modeling \& Simulation}, 16(1):385--411, 2018.
	
	\bibitem{peterseim2016robust}
	D.~Peterseim and R.~Scheichl.
	\newblock Robust numerical upscaling of elliptic multiscale problems at high
	contrast.
	\newblock {\em Computational Methods in Applied Mathematics}, 16(4):579--603,
	2016.
	
	\bibitem{QGVW2017}
	E.~Qian, M.~Grepl, K.~Veroy, and K.~Willcox.
	\newblock A certified trust region reduced basis approach to {PDE}-constrained
	optimization.
	\newblock {\em SIAM Journal on Scientific Computing}, 39(5):S434--S460, 2017.
	
	\bibitem{Quarteroni2016}
	A.~Quarteroni, A.~Manzoni, and F.~Negri.
	\newblock {\em Reduced Basis Methods for Partial Differential Equations},
	volume~92 of {\em La Matematica per il 3+2}.
	\newblock Springer International Publishing, {Cham}, 1 edition, 2016.
	
	\bibitem{SS22}
	J.~Schleu\ss and K.~Smetana.
	\newblock Optimal local approximation spaces for parabolic problems.
	\newblock {\em Multiscale Model. Simul.}, 20(1):551--582, 2022.
	
	\bibitem{SP2016}
	K.~Smetana and A.~T. Patera.
	\newblock Optimal local approximation spaces for component-based static
	condensation procedures.
	\newblock {\em SIAM J. Sci. Comput.}, 38(5):A3318--A3356, jan 2016.
	
	\bibitem{MR3969255}
	F.~Wein, N.~Chen, N.~Iqbal, M.~Stingl, and M.~Avila.
	\newblock Topology optimization of unsaturated flows in multi-material porous
	media: application to a simple diaper model.
	\newblock {\em Commun. Nonlinear Sci. Numer. Simul.}, 78:104871, 16, 2019.
	
	\bibitem{EEnq}
	E.~Weinan, B.~Engquist, and Z.~Huang.
	\newblock Heterogeneous multiscale method: a general methodology for multiscale
	modeling.
	\newblock {\em Physical Review B}, 67(9):092101, 2003.
	
	\bibitem{YM2013}
	Y.~Yue and K.~Meerbergen.
	\newblock Accelerating optimization of parametric linear systems by model order
	reduction.
	\newblock {\em SIAM Journal on Optimization}, 23(2):1344--1370, 2013.
	
\end{thebibliography}

}

\end{document}